\documentclass[aop,MSNbibl,dvips]{arximspdf}
\usepackage{mathbh}
\usepackage{accents}

\doi{10.1214/13-AOP844} %kopijuoti is PTS
\volume{42}
\issue{6}
\pubyear{2014}
\firstpage{2207}
\lastpage{2242}

\makeatletter
\newcommand{\rrvert}{\vert}
\newcommand{\llvert}{\vert}
\renewcommand{\mathring}[1]{\accentset{\circ}{#1}}
\newcommand{\cir}[1]{\mathring{#1}}
\newtheorem{thrm}{Theorem}[section]
\newtheorem{lemma}[thrm]{Lemma}
\newtheorem{cor}[thrm]{Corollary}
\newtheorem{prop}[thrm]{Proposition}
\newproclaim{defi}[thrm]{Definition}
\newproclaim{remark}[thrm]{Remark}

\newcommand{\e}{\varepsilon}

\renewcommand{\o}{\omega}

\newcommand{\w}{\omega}
\makeatother

\begin{document}
\begin{frontmatter}

\title{Local universality of repulsive particle systems and random matrices}
\runtitle{Local universality of repulsive particle systems}

\begin{aug}
\author[1]{\fnms{Friedrich} \snm{G\"otze}\corref{}\ead[label=e1]{goetze@math.uni-bielefeld.de}\thanksref{T1}}
\and
\author[1]{\fnms{Martin} \snm{Venker}\ead[label=e2]{mvenker@math.uni-bielefeld.de}\thanksref{T2}}
\runauthor{F. G\"otze and M. Venker}
\affiliation{University of Bielefeld}
\address[1]{Department of Mathematics\\
University of Bielefeld\\
Postbox 100131\\
33501 Bielefeld\\
Germany\\
\printead{e1}\\
\phantom{E-mail:\ }\printead*{e2}}
\end{aug}
\thankstext{T1}{Supported by CRC 701.}
\thankstext{T2}{Supported by IRTG 1132 and CRC 701.}

% HISTORY:
\received{\smonth{8} \syear{2012}}
\revised{\smonth{3} \syear{2013}}

% ABSTRACT
%
\begin{abstract}
We study local correlations of certain interacting particle systems on
the real line which show repulsion
similar to eigenvalues of random Hermitian matrices.
%More precisely the pair interaction differs from the GUE-eigenvalue
%interaction by a smooth function.
Although the new particle system does not seem to have a
natural spectral or
determinantal representation, the local correlations in the bulk
coincide in the limit of infinitely many particles with those known
from random Hermitian matrices; in particular they can be expressed as
determinants of the so-called sine kernel.
These results may provide an explanation for the appearance of sine
kernel correlation statistics in a number of situations
which do not
have an obvious interpretation in terms
of random matrices.
\end{abstract}

% KEYWORDS
\begin{keyword}[class=AMS]
\kwd[Primary ]{15B52}
\kwd{60B20}
\kwd[; secondary ]{82C22}
\end{keyword}
\begin{keyword}
\kwd{Universality}
\kwd{sine kernel}
\kwd{random matrices}
\kwd{repulsive particles}
\end{keyword}

\end{frontmatter}

%s1 #&#
\section{Introduction and main results}\label{sec1}
This paper is motivated by the surprising emergence of sine kernel
statistics in many
%seemingly unrelated models and
real world observations such as %including examples as zeros of
%$L$-functions, %nuclear level resonances,
parking cars,
perching birds on lines and so on. In the field of random matrices, the
sine kernel describes the local correlations of eigenvalues
in
the bulk of the spectrum of Hermitian random matrices. There it has
been shown to be universal to a high extent; that is,
% being
%unaffected by many changes
%affected only by the first four? moments of the distribution of the
%matrix entries.
it appears for many essentially different matrix distributions.
%However, (known) random matrix models do not seem to provide an
%explanation for the universality of the sine kernel beyond random
%matrix theory.
In this article we % conjecture
show that the sine kernel
%in fact
describes the local correlations of more general repulsive particle
systems on the real line
which only share the repulsion strength exponent $\beta=2$ with the
eigenvalues of (unitary invariant) Hermitian random matrices.
We expect that this behavior extends to larger classes of invariant
ensembles of random matrices, with repulsion
exponents $\beta$ different from two.

To formulate our results, let us recall the so-called invariant $\beta$-ensembles from random matrix theory. Given a continuous
function $Q\dvtx  \mathbb{R} \longrightarrow \mathbb{R}$ of
sufficient growth at infinity and $\beta>0$, set
%
%e1 #&#
\begin{equation}
P_{N,Q,\beta}(x):=\frac{1}{Z_{N,Q,\beta}}\prod_{i<j}
\llvert x_i-x_j\rrvert^\beta
e^{- N\sum_{j=1}^NQ(x_j)}.\label{p40}
\end{equation}
(With a slight abuse of notation, we will not distinguish between a
measure and its density.)
For the ``classical values'' $\beta=1,2,4$, $P_{N,Q,\beta}$ is the
eigenvalue distribution of a probability ensemble on the space of
$(N\times N)$ matrices with real symmetric ($\beta=1$), complex
Hermitian ($\beta=2$) or quaternionic self-dual ($\beta=4$) entries,
respectively. For arbitrary $\beta$, only for quadratic $Q$,
$P_{N,Q,\beta}$ is known to be an eigenvalue distribution.

The notion of bulk universality is usually formulated via the
correlation functions of the ensemble. For\vspace*{1pt} a probability measure
$P_N(x)\,dx$ on
$\mathbb{R}^N$ and \mbox{$k=1,2,\ldots,N$}, the \textit{$k$th correlation
function} $\rho_N^k\dvtx  \mathbb{R}^k \longrightarrow \mathbb{R}$
of $P_N$ is defined as
\[
\rho_N^k(x_1,\ldots,x_k):=\int
_{\mathbb{R}^{N-k}}P_N(x)\,dx_{k+1}\cdots
dx_N.
\]
The correlation functions $\rho_N^k$ are the densities of the
marginals of $P_N$.
The measure
$\rho_N^k(t)\,dt$ on $\mathbb{R}^k$ is called \textit{$k$th correlation
measure}.

It is known that under very mild conditions on $Q$, there is an
absolutely continuous probability measure $\mu_{Q,\beta}(t)\,dt$ on
$\mathbb{R}$, which is the weak limit of $\rho_{N,Q,\beta}^1(t)\,dt$
as $N\to\infty$.

Now, $P_{N,Q,\beta}$ is said to admit bulk universality, if for all
$a$ with $\mu_{Q,\beta}(a)>0$ and all
$t_1,\ldots,t_k$ the limit
%
%e2 #&#
\begin{equation}
\lim_{N\to\infty}\frac{1}{\mu_{Q,\beta}(a)^k}\rho_{N}^k
\biggl( a+\frac{t_1}{N\mu_{Q,\beta}(a)},\ldots,a+\frac{t_k}{N\mu_{Q,\beta
}(a)} \biggr)\label{p39}
\end{equation}
exists and coincides with the one for $P_{N,G,\beta}$, $G$ quadratic
(the so-called Gaussian $\beta$-ensemble). Universality here
should be understood as
%is justified as existence of the limit \eqref{p39} and
a coincidence of limit (\ref{p39}) with the corresponding Gaussian
$\beta$-ensemble.
This has been %proven
established for
large classes of $Q$. The scaling in (\ref{p39}) %is made
is chosen such that the
% to have an
asymptotic mean spacing between consecutive eigenvalues is normalized
to $1$.
However, it is known that the limit
depends on $\beta$.

In the case $\beta=2$, which appears frequently in ``real world
statistical studies,''
%real world
the limiting object (\ref{p39}) %has a very easy expression:
is determinantal of type
%
%e3 #&#
\begin{eqnarray}\label{p42}
&& \lim_{N\to\infty}\frac{1}{\mu_{Q,2}(a)^k}\rho_{N}^k
\biggl( a+\frac{t_1}{N\mu_{Q,2}(a)},\ldots,a+\frac{t_k}{N\mu
_{Q,2}(a)} \biggr)
\nonumber\\[-8pt]\\[-8pt]
&&\qquad =\det\biggl[\frac{\sin(\pi(t_i-t_j) )}{\pi(t_i-t_j)} \biggr]_{1\leq
i,j\leq k},\nonumber
\end{eqnarray}
%
%The kernel inside the determinant is the famous sine kernel.
involving the sine kernel
\[
\mathbb{S}(t):=\frac{\sin(\pi t)}{\pi t},\qquad t\neq0, \mathbb{S}(0):=1.
\]
Universality for unitary invariant ensembles, that is, $\beta=2$
invariant ensembles, was proved in many papers, for example
(naming only
few)
\cite{PasturShcherbina97,PasturShcherbina,Deiftetal99,McLaughlinMiller08,LevinLubinsky08}. Recently universality
(for general
$\beta$-ensembles) was proved in \cite{Bourgadeetal1,Bourgadeetal}.
For $\beta=1,2$, bulk universality was also proved for Wigner
matrices by two
groups of authors.
% which are briefly speaking random symmetric/Hermitian matrices which
%have as much independence in their entries as possible due to
%the
%symmetry constraints. Proving the bulk universality for Wigner
%matrices has been a long-standing problem and was only
%solved recently by two
%different groups of authors.
%Basing
Based on earlier work of Johansson \cite{Johansson01}, universality
was shown for general classes
of Wigner matrices in a series of papers by Erd{\H{o}}s, Yau, Schlein,
Yin, Ramirez and Peche (see \cite{ErdosYau12} for a survey on
their results) and Tao and Vu; see \cite{TaoVu12} for a survey on
their results. We remark that bulk universality was proved in
\cite{GotzeGordin} for the Hermitian fixed trace ensemble, a random
matrix which is neither a Wigner matrix nor determinantal.

Writing the density (\ref{p40}) in the Gibbsian form
%
%e4 #&#
\begin{equation}
P_{N,Q,\beta}=\frac{1}{Z_{N,Q,\beta}} e^{\beta\sum_{i<j}\log\llvert
x_i-x_j\rrvert- N\sum_{j=1}^NQ(x_j)},\label{p41}
\end{equation}
we see that $P_{N,Q,\beta}$ can be interpreted as an interacting
particle system on $\mathbb{R}$ in an external field, interacting via a
2d Coulomb potential.

It is believed that many complicated, strongly correlated systems share
the local bulk scaling limit (defined
again by correlation functions) with some random matrix model.
This was conjectured by Wigner who used random matrices to model
energy levels of nuclei.
By the underlying matrix structure, physical requirements (conserved
quantities, time reversal,\,\ldots)
determine the value of $\beta$ in the cases $\beta=1,2,4$.
The limits with $\beta=2$ also seem to appear in statistics of
distances between parking cars \cite{AbulMagd}, waiting times at bus
stops in certain cities
\cite{KrbalekSeba} (see \cite{Baiketal} for a determinantal model)
and the %celebrated
pair
correlation conjecture of Montgomery \cite{Montgomery} for the zeros
of the Riemann Zeta function on the critical line. See, for example,
\cite{KeatingSnaith00} for more relations between the Riemann Zeta
function and random matrix theory.
A common cause for the appearance of sine kernel statistics in a number
of statistics about
real world repulsive systems and in physics and mathematics still
remains to be identified.

We consider here a class of more general interacting particle systems,
defined by the density
%
%e5 #&#
\begin{equation}
\frac{1}{Z_{N,\varphi,Q}}\prod_{i<j}\varphi(x_i-x_j)e^{- N\sum
_{j=1}^NQ(x_j)},\label{p43}
\end{equation}
where $Q$ is a continuous function of sufficient growth at infinity
compared to the continuous function
$\varphi\dvtx  \mathbb{R} \longrightarrow [0,\infty)$.
Apart from some technical conditions we will assume that
%
%e6 #&#
\begin{equation}
\varphi(0)=0,\qquad\varphi(t)>0\qquad\mbox{for }t\neq0\quad\mbox
{and}\quad\lim
_{t\to0}\frac{\varphi(t)}{\llvert t\rrvert^\beta}=c>0,
\end{equation}
or, in other terms, $0$ is the only zero of $\varphi$ and it is of
order $\beta$.

We expect that (at least under some smoothness and growth conditions)
the bulk scaling limit of (\ref{p43}) coincides with that of
the $\beta$-ensembles, since in view of the regular local distribution
of eigenvalues/particles at $1/N$ spacings only the exponents
of the
interaction kernel should determine the local universality class.

The purpose of this paper is to prove this for $\beta=2$ and a special
class of $\varphi$~and~$Q$.
From now on, we will always deal with the case $\beta=2$, therefore
omitting the subscript $\beta$.
To state our results, let $h$ be a continuous even function which is
bounded below. Let $Q$ be
a continuous even function of sufficient growth at infinity. By
$P_{N,Q}^h$ we will denote the probability density on $\mathbb{R}^N$ defined
by
%
%e7 #&#
\begin{equation}
\quad P_{N,Q}^h(x):=\frac{1}{Z_{N,Q}^h}\prod
_{i<j}\llvert x_i-x_j\rrvert
^2\exp\Biggl\{- N\sum_{j=1}^NQ(x_j)-
\sum_{i<j}h(x_i-x_j)
\Biggr\}, \label{211}
\end{equation}
where $Z_{N,Q}^h$ denotes the normalizing constant. The density
$P_{N,Q}^h$ can also be written in the form (\ref{p43}) with
$\varphi(t):=t^2\exp\{-h(t)\}$.
The first result describes the global scaling limit of the correlation
measures of $P_{N,Q}^h$. To formulate it, introduce
for a twice differentiable convex function $Q$ the quantity $\alpha
_Q:=\inf_{t\in\mathbb{R}}Q''(t)$.
Moreover, denote by $\rho_{N,Q}^{h,k}$ the $k$th correlation
function of $P_{N,Q}^h$.

%th1.1 #&#
\begin{thrm}\label{300}
Let $h$ be a real analytic and even Schwartz function. Then there
exists a constant $\alpha^h\geq0$ such that for all real analytic,
strictly convex and even $Q$ with $\alpha_Q> \alpha^h$, the following
holds:

There exists a compactly supported probability measure $\mu_Q^h$
having a nonzero and continuous density on the interior of its
support and for
$k=1,2,\ldots,$ the $k$th correlation measure of $P_{N,Q}^h$ converges
weakly to the $k$-fold product $ (\mu_Q^h )^{\otimes k}$,
that is, for any
bounded and continuous function $g\dvtx  \mathbb{R}^k \longrightarrow
\mathbb{R}$,
%
%e8 #&#
\begin{equation}\label{p33}
\lim_{N\to\infty}\int g \rho_{N,Q}^{h,k}
\,d^kt=\int g\, d \bigl(\mu_Q^h
\bigr)^{\otimes k}. % \int g(t_1,\ldots,t_k)\rho_{N,Q}^{h,k}(t_1,
\end{equation}
\end{thrm}

\begin{remark*}
(a) If $h$ is\vspace*{1pt} (additionally) positive semi-definite, then
$\alpha^h$ in Theorem~\ref{300} may be explicitly chosen as
$\alpha^h=\sup_{t\in\mathbb{R}}-h''(t)$.

(b) In general, the measure $\mu_Q^h$ depends on $h$.

(c) $P_{N,Q}^h$ does not seem to be either determinantal nor
have a natural spectral interpretation; therefore we will speak of
particles instead of eigenvalues.
% \item The assumptions on $h$ can be weakened. Inspecting the proof of
%Theorem~\ref{300} shows that the same result holds for all
% real-analytic and even $f\in L^1(\R)$ such that $h''$ is bounded
%below. This remark extends to the next theorem.

(d) We remark that in \cite{BoutetdeMonvelPasturShcherbina},
macroscopic correlations have been
studied in a more general setup.
\end{remark*}

The next result states the universality of the sine kernel in the local
scaling limit in the bulk.

%
%th1.2 #&#
\begin{thrm}\label{312}
Let $h$ and $Q$ satisfy the assumptions of Theorem~\ref{300}. Then for
$k=1,2,\ldots,$ we have
%
%e9 #&#
\begin{eqnarray}\label{325}
&&\lim_{N\to\infty}\frac{1}{\mu_Q^h(a)^k}\rho_{N,Q}^{h,k}
\biggl(a+\frac{t_1}{N\mu_Q^h(a)},\ldots,a+\frac{t_k}{N\mu_Q^h(a)} \biggr)
\nonumber\\[-8pt]\\[-8pt]
&&\qquad =\det\biggl[\frac{\sin(\pi(t_i-t_j) )}{\pi(t_i-t_j)} \biggr]_{1\leq
i,j\leq k}\nonumber
\end{eqnarray}
uniformly in $t_1,\ldots,t_k$ from any compact subset of $\mathbb
{R}^k$ and uniformly in the point $a$ from any
compact proper subset of the support of $\mu_Q^h$.
\end{thrm}

\begin{remark*}
(a) If $h$ is positive semi-definite, then $\alpha^h$ in
Theorem~\ref{312} may be explicitly chosen as
$\alpha^h=\sup_{t\in\mathbb{R}}-h''(t)$.\vadjust{\goodbreak}

(b) Bulk universality for ensembles of form (\ref{211}) with
arbitrary
$\beta>0$ replacing the repulsion exponent 2 in
(\ref{211}) has been shown by the second author in \cite{Venker12}.
The notion of universality is weaker than in the present
paper. The proof of bulk universality uses methods similar to the
present work, combined with techniques developed by Erd\H{o}s, Yau
and co-workers; see, for example, \cite{ErdosYau12} for a review.

(c) Similar results hold at the edge of the support of $\mu
_Q^h$. An article on edge universality
of $P_{N,Q}^h$ is in preparation \cite{KriecherbauerVenker}.
\end{remark*}

%As an introductory example,
We shall demonstrate our approach to bulk universality by means of the
following example of functions $h$ and $Q$.

%th1.3 #&#
\begin{thrm}\label{313}
Let $\gamma>0$ and $\alpha>0$ be arbitrary.
Let $h(t-s):=\gamma(t-s)^2$ and $Q(t)=\alpha
t^2$. Then (\ref{p33}) and (\ref{325}) hold for $
(P_{N,Q}^h )_N$ uniformly as in Theorem~\ref{312}. Here $\mu
_Q^h$ will be the
semi-circle
distribution with support
$[-\omega,\omega]$, $\omega:= (\sqrt{\alpha+\gamma} )^{-1}$.
\end{thrm}

A first step in the proof of Theorems~\ref{300}~and~\ref{312} is to
compare the correlation functions of $P_{N,Q}^h$ with\vspace*{-2pt}
correlation functions of \textit{eigenvalues} of
some unitary invariant ensemble. To construct such an ensemble, we
first determine
$\mu_Q^h$ as the equilibrium measure of some external field $V$
(depending on $h$ and $Q$) using a fixed point argument.
The difference between $P_{N,Q}^h$ and this unitary invariant ensemble
$P_{N,V}$ consists of (up to normalization) a factor
$\exp\{\mathcal{U}(x)\}$,
where $\mathcal{U}$ is a~quadratic interaction energy which may be expressed
as a mixture of linear interaction energy terms using Gaussian
processes. This finally leads, after a truncation procedure,
to a mixture representation of $P_{N,Q}^h$
by invariant ensembles with the \textit{same} bulk universality.

The paper is organized as follows. In Section~\ref{firstex}, the asymptotics of
$P_{N,Q}^h$ for $h(t-s):=\gamma(t-s)^2$ and $Q(t)=\alpha
t^2$ are investigated, and in particular Theorem~\ref{313} is proved.
In Section~\ref{correns}, we associate to $P_{N,Q}^h$ a unitary\vspace*{-2pt} invariant
ensemble which will turn out to have the same asymptotic behavior as
$P_{N,Q}^h$. Section~\ref{globasym}\vspace*{1.5pt} contains concentration of
measure inequalities. Section~\ref{sec5} deals with bounds on the first
correlation function of a unitary
invariant ensemble.
The proofs in this section use established techniques which we decided
to include in detail for the sake of completeness of the
exposition.
Theorems~\ref{300}~and~\ref{312} are proved in Section~\ref{sec6}. In the
\hyperref[equmeas]{Appendix} we recall a number of results on equilibrium measures.

A prior version of these results is based on the Ph.D. thesis of the
second author~\cite{Venker}.

%s2 #&#
\section{A first example}\label{firstex}
In this section, we will study the probability measure
%
%e10 #&#
\begin{eqnarray}
&& P_N^{\alpha,\gamma}(x)
\nonumber
\\[-8pt]
\\[-8pt]
\nonumber
&&\qquad:=\frac{1}{Z^{\alpha,\gamma}_N}\prod
_{1\leq
i<j\leq N}\llvert x_i-x_j\rrvert
^2\exp\biggl\{-\alpha N M_2(x) %\sum_{j=1}^Nx_j^2
-\gamma\sum
_{i<j}(x_i-x_j)^2
\biggr\},
\end{eqnarray}
using the potentials $M_p(x):=\sum_{j=1}^Nx_j^p$ with $p=2$ and
constants $\alpha,\gamma>0$, where $Z^{\alpha,\gamma}_N$ denotes
the normalization
factor. In the following we shall suppress the dependencies on $\alpha
$ and $\gamma$.

We will reduce bulk universality of $ (P_N^{\alpha,\gamma}
)_N$ to the well-known bulk universality of the GUE.

It is convenient to
introduce the distribution GUE$_{\omega}$, depending on a parameter
$\omega>0$, as
\[
P_{N,\omega}^{\mathrm{GUE}}(x):=\frac{1}{Z_{N,\w}^{\mathrm{GUE}}}\prod
_{j<k}\llvert x_k-x_j\rrvert
^2 \exp\biggl\{-\frac
{2}{\omega^2} N M_2(x) %\sum_{j=1}^N x_j^2
\biggr\}.
\]
Under this scaling the first correlation measure of $P_{N,\omega
}^{\mathrm{GUE}}$ will
converge to the semicircle law
supported on $[-\omega,\omega]$; for a proof see, for example, \cite
{Pastur99}.
First we rewrite the density $P_N:=P_N^{\alpha,\gamma}$ using
%
%e11 #&#
\begin{eqnarray}
\qquad \gamma\sum_{i<j}(x_i-x_j)^2 &=& \gamma N M_2(x) %\sum_{j=1}^Nx_j^2
-\gamma M_1(x)^2
%(\sum_{j=1}^Nx_j)^2
\quad\mbox{as}\nonumber
\\
P_N(x)&=&\frac{1}{Z_N}\prod
_{1\leq i<j\leq N}\llvert x_i-x_j\rrvert ^2
\\
&&\hspace*{56pt}{}\times \exp\bigl\{-(\alpha+\gamma) N M_2(x) %\sum_{j=1}^Nx_j^2
+\gamma M_1(x)^2 \bigr\}.\nonumber %(\sum_{j=1}^Nx_j)^2\}.
\end{eqnarray}
Using the simple identity
%
%e12 #&#
%e13 #&#
\begin{eqnarray}
\exp\bigl\{{\gamma t^2} \bigr\}&=&\frac{1}{2\pi}\int
_\mathbb{R}\exp\{\e\sqrt{\gamma}t \}\exp\bigl\{-
\e^2/4 \bigr\}\,d\e, \qquad\mbox{we may write}
\\
P_N(x)&=&\frac{1}{2\pi}\int_\mathbb{R}
\frac{1}{Z_N}\prod_{1\leq
i<j\leq N}\llvert
x_i-x_j\rrvert^2\nonumber
\\
&&\hspace*{87pt}{}\times\exp\bigl\{-(\alpha+\gamma) N M_2(x) +\sqrt{\gamma}\e M_1(x) \bigr\}\nonumber\hspace*{-50pt}
\\
&&\hspace*{87pt}{}\times \exp\bigl\{{-\e^2/4} \bigr\}\,d\e\nonumber
\\
%&=&\frac{1}{2\pi}\int_\R\frac{Z_N^\e}{Z_N}\frac{1}{Z_N^\e}\prod_{1\leq
%%i<j\leq N}\lv x_i-x_j\rv^2\exp\{-(\a+\g)
%N M_2(x)
%+\sqrt{\g}\e\sum_{j=1}^Nx_j\}e^{-\e^2/4}\,d\e,\nonumber\\
&=&\frac{1}{2\pi}\int_\mathbb{R}\frac{Z_N^\e}{Z_N}P_N^{\e
}(x)e^{-\e^2/4}\,d \e\qquad\mbox{where}\label{mixequation}
\\
P_N^{\e}(x) &:=&\frac{1}{Z_N^\e}\prod
_{1\leq i<j\leq N}\llvert x_i-x_j\rrvert ^2\nonumber
\\
&&\hspace*{57pt}{}\times \exp\bigl\{-(\alpha+\gamma) N M_2(x) %\sum_{j=1}^Nx_j^2
+ \sqrt{\gamma}\e M_1(x) %\sum_{j=1}^Nx_j
\bigr\},\nonumber
\\
Z_N^\e&:=&\int_{\mathbb{R}^N}\prod
_{1\leq i<j\leq N}\llvert x_i-x_j\rrvert^2\nonumber
\\
&&\hspace*{57pt}{}\times \exp\bigl\{-(\alpha+\gamma) N M_2(x) %\sum_{j=1}^Nx_j^2
+\sqrt{\gamma}\e M_1(x) %\sum_{j=1}^Nx_j
\bigr\}\,dx.\nonumber\hspace*{-50pt}
\end{eqnarray}
We have thus expressed $P_N$ as a probabilistic mixture of the
probability measures~$P_{N}^\e$.

The next lemma deals with the ratio $Z_N^\e/Z_N$.
%
%le2.1 #&#
\begin{lemma} \label{prop4}
For each $\e$, each $N$ and all $\alpha,\gamma>0$ we have
\[
Z_N^\e/Z_N=\exp\biggl\{\frac{\gamma\e^2}{4(\alpha+\gamma)}
\biggr\} \biggl(\sqrt{1-\frac{\gamma}{\alpha+\gamma}} \biggr)^{-1}.
\]
\end{lemma}

\begin{pf}
We first expand the fraction
\[
Z_N^\e/Z_N=\bigl(Z_N^\e/Z_{N,\omega}^{\mathrm{GUE}}
\bigr)/\bigl({Z_N/Z_{N,\w
}^{\mathrm{GUE}} }\bigr)\qquad
\mbox{where }\omega=(\alpha+\gamma)^{-1/2}.
\]
The diagonal elements of a GUE$_{\omega}$ matrix are independent
Gaussians with
mean $0$ and variance $\frac{1}{2N(\alpha+\gamma)}$. Using this, we
get easily
for any $\e$, any $N$ and any $\alpha,\gamma>0$
\[
Z_N^\e/Z_{N,\w}^{\mathrm{GUE}}=
\mathbb{E}_{N,\mathrm{GUE}_\omega}\exp\bigl\{\e\sqrt{\gamma
}M_1(x)\bigr\}=
\exp\bigl\{\gamma\e^2\cdot\bigl(4(\alpha+\gamma)\bigr)^{-1}
\bigr\},
\]
where $\mathbb{E}_{N,\mathrm{GUE}_\omega}$ denotes expectation w.r.t.
$P_{N,\omega}^{\mathrm{GUE}}$. Similarly, we get for any $N$ and any
$\alpha,\gamma>0$
\[
{Z_N}/{Z_{N,\w}^{\mathrm{GUE}}}=\mathbb{E}_{N,\mathrm{GUE}_\omega}
\exp\bigl\{\gamma M_1(x)^2 \bigr\}= \bigl(1-\gamma/(
\alpha+\gamma) \bigr)^{-1/2}.
\]\upqed
\end{pf}
%
%
%de2.2 #&#
\begin{defi}\label{defi5}
For $\omega>0$, the probability measure $\sigma_\omega$ on $\mathbb
{R}$ given by
\[
\sigma_\omega(t)\,dt:=\frac{2}{\pi\omega^2}\sqrt{\omega^2-t^2}
\mathbh{1}_{[-\o,\o]}(x)\,dt
\]
is called (\textit{Wigner's}) \textit{semicircle law} (with parameter $\omega$).
\end{defi}
%
%For the remainder of this section, we set $\omega:=(\a+\g)^{-1/2}$.
By equation (\ref{mixequation}), $P_N$ is a mixture of $P_N^\e$. We
show first that the statement of Theorem~\ref{313} is true
for
each $\e\in\mathbb{R}$ if we replace $P_{N,Q}^h$ by $P_N^\e$.
Eventually we will use Lebesgue's dominated
convergence theorem.

%
%pr2.3 #&#
\begin{prop} \label{etaprop}
Let $\rho_N^{k,\e}$ denote the $k$th correlation function of $P_N^\e
$ and set~$\omega=\sqrt{\frac{1}{\alpha+\gamma}}$.
\begin{longlist}[(a)]
\item[(a)] For any $\e\in\mathbb{R}$, any $k$ and any continuous,
bounded $g\dvtx  \mathbb{R}^k \longrightarrow \mathbb{R}$ we have
\[
\lim_{N\to\infty}\int_{\mathbb{R}^k} g\,d \rho_N^{k,\e}=\int_{[-\omega,\omega]^k} g\, d(
\sigma_\omega)^k.
\]

\item[(b)] We have for any $\e$ and any $k$,
\begin{eqnarray*}
&&\lim_{N\to\infty}\frac{1}{\sigma_\omega(a)^k}\rho_N^{k,\e
}\biggl(a+\frac{t_1}{N\sigma_\omega(a)},\ldots,a+\frac{t_k}{
N\sigma_\omega(a) } \biggr)
\\
&&\qquad =\det\bigl(\mathbb{S}(t_i-t_j)\bigr)_{1\leq i,j\leq k}
\end{eqnarray*}
locally uniformly for all $t_1,\ldots, t_k$ and uniformly for $a$ varying
in a compact subset of $(-\omega,\omega)$.
\end{longlist}
\end{prop}

\begin{pf}
A proof of the first part can be found in \cite{Johansson98}. For the
second part we use orthogonal
polynomials. Note that the polynomials orthogonal to a Gaussian weight
with nonzero mean\vspace*{-1pt}
are normalized shifted Hermite polynomials. Let $\pi^{\mathrm{(N)}}_j$
denote the $j$th Hermite
polynomial orthonormal w.r.t. the weight $e^{-N(\alpha+\gamma)t^2}$.

It is easy to check that the set of polynomials orthogonal w.r.t. the weight
$e^{-N(\alpha+\gamma)t^2+\e\sqrt{\gamma}t}$
are the polynomials $ (\pi^{\mathrm{(N)}*}_j )_j$, where
%
%e14 #&#
\begin{equation}\label{206}
\pi^{\mathrm{(N)}*}_j(t):=e^{(\o''\e^2/2N)}\pi^{\mathrm{(N)}}_j
\bigl(t-\o'\e/2N\bigr)
\end{equation}
with $\o':=\sqrt{\gamma}/(\alpha+\gamma)$ and $\o'':={\o'}^2/4$.
The ensemble $P_N^\e$ is determinantal, that~is,
%
%e15 #&#
\begin{equation}
\rho_{N}^{k,\e}(t_1,\ldots,t_k)=(N-k)!/(N!)
\det\bigl(K_N^*(t_i,t_j)
\bigr)_{i,j=1}^k,\label{218}
\end{equation}
where $K_N^*(t,s)=\sum_{j=0}^{N-1}\pi^{\mathrm{(N)}*}_j(t)\pi
^{\mathrm{(N)}*}_j(s)$. From (\ref{206}) we get
%
%e16 #&#
\begin{equation}
K_N^*(t,s)=e^{(\o''\e^2)/N}K_N\bigl(t-\o'
\e/2N,s-\o'\e/2N\bigr),\label{214}
\end{equation}
where $K_N$ denotes the kernel corresponding to the ensemble
$P_{N,\omega}^{\mathrm{GUE}}$.
Hence we have
%
%e17 #&#
\begin{eqnarray}\label{215}
&&\frac{1}{\sigma_\omega(a)}K_N^* \biggl(a+\frac{t}{N\sigma_\omega
(a)},a+\frac{s}{N\sigma_\omega(a)} \biggr)\nonumber
\\
&&\qquad =\frac{e^{(\o''\e^2)/N}}{\sigma_\omega(a)}K_N \biggl( a+\frac{t-\o'\e
\sigma_\omega(a)/2}{N\sigma_\omega(a)},a+
\frac{s-\o'\e\sigma_\omega(a)/2}{N\sigma_\omega(a)} \biggr)
\\
&&\qquad =\frac{e^{(\o''\e^2)/N}}{\sigma_\omega(a)}K_N \biggl( a+\frac
{t'}{N\sigma_\omega(a)},a+
\frac{s'}{N\sigma_\omega(a)} \biggr),\nonumber
\end{eqnarray}
where $t':=t-\o'\e\sigma_\omega(a)/2$ and $s':=s-\o'\e\sigma
_\omega(a)/2$. It is well known that
%
%e18 #&#
\begin{equation}
\lim_{N\to\infty}\frac{1}{\sigma_\omega(a)}K_N \biggl( a+
\frac{t'}{N\sigma_\omega(a)},a+\frac{s'}{N\sigma_\omega(a)} \biggr
)=\frac{\sin(\pi(t'-s'))}{\pi(t'-s')}.\label{216}
\end{equation}
For a proof of (\ref{216}) see, for example, \cite{Deift98}, Chapter~8, or Theorem~\ref{thrmLubinsky}.
Since
${\lim_{N\to\infty}\exp\{{(\o''\e^2)/N} \}=1,}$ we get
from (\ref{215}) and (\ref{216}) that
%
%e19 #&#
\begin{eqnarray}\label{217}
&& \lim_{N\to\infty}\frac{1}{\sigma_\omega(a)}K_N^* \biggl( a+
\frac{t}{N\sigma_\omega(a)},a+\frac{s}{N\sigma_\omega(a)} \biggr)
\nonumber\\[-8pt]\\[-8pt]
&&\qquad =\frac
{\sin(\pi(t'-s'))}{\pi(t'-s')}
=\frac{\sin(\pi(t-s))}{\pi(t-s)}.\nonumber
\end{eqnarray}
Now, by (\ref{217}) and (\ref{218}), the second assertion of
Proposition~\ref{etaprop} follows. As (\ref{216}) is true locally
uniformly in $t',s'$ and uniformly in $a\in I$, $I\subset[-\omega,\omega]$ compact, we get (\ref{217}) locally uniformly in $t,s$ and
uniformly in $a\in I$.
\end{pf}

\begin{pf*}{Proof of Theorem~\ref{313}}
By equation (\ref{mixequation}) and Lemma~\ref{prop4} we know that
%
%e20 #&#
\begin{equation}
\label{219} P_N(x)=\int_\mathbb{R}P_N^\e(x)p(
\e)\,d\e,
\end{equation}
where $p$ is an $N$-independent probability measure on $\mathbb{R}$.
Using Fubini's theorem,~(\ref{219}) implies
$\int_{\mathbb{R}^k} g \,d\rho_N^k=\int_\mathbb{R}\int_{\mathbb
{R}^k} g \,d\rho_N^{k,\e} p(\e)\,d\e$ and $\rho_N^k(t_1,\ldots,t_k)=\int_\mathbb{R}
\rho_N^{k,\e}(t_1,\ldots,t_k)p(\e)\,d\e$,
and hence for each compact $K\subset\mathbb{R}^k$ and each compact
$I\subset(-\omega,\omega)$
%
%e21 #&#
\begin{eqnarray}\label{p34}
&&\sup_{t\in K,a\in I}\biggl\llvert\sigma_\omega(a)^{-k}
\rho_N^k\biggl(a+\frac{t_1}{N\sigma_w(a)},\ldots,a+
\frac{t_k}{N\sigma_w(a)}\biggr)\nonumber
\\
&&\hspace*{131pt}{} -\det\bigl(\mathbb{S}(t_i-t_j)
\bigr)_{1\leq i,j\leq k}\biggr\rrvert \nonumber
\\
&&\qquad =\sup_{t\in K,a\in I}\biggl\llvert\int_\mathbb{R}p(
\e) \biggl( \sigma_\omega(a)^{-k}\rho_N^{k,\e}
\biggl(a+\frac{t_1}{N\sigma
_w(a)},\ldots,a+\frac{t_k}{N\sigma_w(a)}\biggr)
\nonumber\\[-8pt]\\[-8pt]
&&\hspace*{209pt}{}  -\det\bigl(\mathbb{S}(t_i-t_j) \bigr)_{1\leq
i,j\leq k} \biggr)\,d\e \biggr\rrvert\nonumber\hspace*{-15pt}
\\
&&\qquad \leq\int_\mathbb{R}p(\e)\sup_{t\in
K,a\in I}\biggl\llvert\sigma_\omega(a)^{-k}\rho_N^{k,\e}
\biggl(a+\frac{t_1}{N\sigma_w(a)},\ldots,a+\frac{t_k}{N\sigma_w(a)}\biggr)\nonumber
\\
&&\hspace*{204pt}{}  -\det\bigl(
\mathbb{S}(t_i-t_j) \bigr)_{1\leq i,j\leq k}\biggr\rrvert \,d\e,\nonumber\hspace*{-15pt}
\end{eqnarray}
where we stick to the notation of Proposition~\ref{etaprop}. Theorem
\ref{313} will follow from
Proposition~\ref{etaprop} if $\int_{\mathbb{R}^k} g \,d\rho_N^{k,\e
}$ and $\sup_{t\in K,a\in I}\llvert\rho_N^{k,\e}(s_1,\ldots,s_k)\rrvert$,
$s_i:=a+t_i/(N\sigma_\omega(a))$, are
uniformly
bounded in $\e$. The
uniform boundedness of
$\int_{\mathbb{R}^k} g \,d\rho_N^{k,\e}$
is immediate as $g$ is bounded.

To show uniform boundedness of $\rho_N^{k,\e}(s_1,\ldots,s_k)$
uniformly in $\e$, $t$ and $a$, we proceed as in the paper by Pastur
and
Shcherbina
\cite{PasturShcherbina}. Since all correlation functions are
nonnegative, we see by Sylvester's criterion from the determinantal
relations (\ref{218}) that
the matrix $(K_{N}^*(t_i,t_j))_{1\leq i,j\leq k}=:A$ is positive
semi-definite and can hence be
written as $A=B^2$ for some matrix $B$. Now using Hadamard's inequality
we get
\[
\det A = (\det B )^2\leq\prod_{j=1}^k
\sum_{i=1}^k\llvert B_{ij}
\rrvert^2=\prod_{j=1}^kA_{jj}.
\]
In our case this reads
%e22 #&#
\begin{equation}\label{267}
\rho_{N}^{k,\e}(s_1,\ldots,s_k)
\leq(N-k)!/(N!) \prod_{j=1}^kK_N(s_j,s_j)
\leq C^k\prod_{j=1}^k
\rho_N^{1,\e}(s_j),
\end{equation}
where $C$ is a constant such that $C\geq N/(N-k)$.
Using (\ref{206}), we get
\begin{eqnarray*}
\rho_N^{1,\e}(s_j)&=&\frac{1}{N}\sum
_{i=0}^{N-1}\pi^{\mathrm{(N)}*}_i(s_j)^2e^{-N(\alpha+\gamma)s_j^2+\sqrt
{\gamma}\e s_j}
\\
&=&\frac{1}{N}\sum_{i=0}^{N-1}
\pi^{\mathrm{(N)}}_i\bigl(t-\o'\e/2N
\bigr)^2e^{-N(\alpha+\gamma)(s_j-\o'\e/2N)^2}
\\
&=& \rho_N^{1,\mathrm{GUE}_\omega}\bigl(s_j-\o'\e/2N\bigr),
\end{eqnarray*}
where $\rho_N^{1,\mathrm{GUE}_\omega}$ is the first
correlation function of the $\mathrm{GUE}_\omega$. From Proposition~\ref{etaprop}(b) for
$k=1,\e=0$ we get that $\rho_N^{1,\mathrm{GUE}_\omega}(s_j-\o
'\e/2N)$ converges (locally) uniformly in $t_j$ and $a$ toward the
bounded function $\sigma_\omega(a)$, hence there\vspace*{1pt} is a constant~$C'$
such that for all $N$ and all $t\in K,a\in I$ we have
\mbox{$\rho_N^{1,\mathrm{GUE}_\omega}(s_j-\o'\e/2N)\leq C'$}.
To see the required uniformity in $\varepsilon$ , either adapt the arguments
in Section~\ref{sec6} following (\ref{p35}) or use
that
$\rho_N^{1,\mathrm{GUE}_\w}(s)$ is
bounded uniformly in $N$ and $s$, as can
be seen from its determinantal
representation
and the well known asymptotics for the
Hermite polynomials.
This
estimate together with~(\ref{267}) finishes the proof of Theorem~\ref{313}.
\end{pf*}

%s3 #&#
\section{The associated random matrix ensemble}\label{correns}
In this section, we start with the investigation of our main model. Let
$h$ be a continuous even function and $Q$ a~strictly convex
symmetric
function and assume that
%
%e23 #&#
\begin{equation}
P_{N,Q}^h(x):=\frac{1}{Z_{N,Q}^h}\prod
_{1\leq i<j\leq N}\llvert x_i-x_j\rrvert
^2e^{-N\sum_{j=1}^NQ(x_j)-\sum_{i<j}h(x_i-x_j)}, \label{modelh}
\end{equation}
defines the density of a probability measure on $\mathbb{R}^N$, where
\[
Z_{N,Q}^h:=\int_{\mathbb{R}^N} \prod
_{1\leq i<j\leq N}\llvert x_i-x_j\rrvert
^2e^{-N\sum_{j=1}^NQ(x_j)-\sum_{i<j}h(x_i-x_j)}\,dx
\]
denotes the normalizing constant. This is, for example, the case if $h$
is bounded below.

We will frequently use the notation
%
%e24 #&#
\begin{equation}\label{notationp}
h_\mu(s):=\int h(t-s)\,d\mu(t),\qquad h_{\mu\mu}:=\int\!\!\int h(t-s)\,d\mu
(t)\,d\mu(s)
\end{equation}
for a compactly supported probability measure $\mu$ on $\mathbb{R}$.
For the statement of the next lemma, $\mathcal{M}_c^1$ will denote the
set of compactly supported (Borel) probability measures
on $\mathbb{R}$.

%
%le3.1 #&#
\begin{lemma}\label{fixedpoint}
Let $h\dvtx  \mathbb{R} \longrightarrow \mathbb{R}$ be even, twice
differentiable, bounded and such that $h''(t)\geq-\alpha_Q$ for all
$t$. Define
$T_h\dvtx  \mathcal{M}_c^1 \longrightarrow \mathcal{M}_c^1$,
$T_h(\mu)$ as the equilibrium measure to the external field $t\mapsto
Q(t)+h_\mu(t)$.

Then\vspace*{1pt} $T_h$ has a fixed point, that is there exists a probability measure $\mu_Q^h$ which
is the equilibrium measure to the external field
$t\mapsto Q(t)+\int h(t-s)\,d\mu_Q^h(s)$.
\end{lemma}

\begin{pf}
We will apply Schauder's fixed point theorem, which states that each
continuous mapping $T\dvtx  C \longrightarrow C$ of a compact, convex and
nonempty subset $C$ of a Hausdorff topological vector space has a
fixed point.

We consider the topological vector space $\mathcal{M}(K)$ of all
signed finite Borel measures on some compact interval $K$ of $\mathbb{R}$,
equipped with the topology of vague convergence. This topology is
metrizable and hence the space is Hausdorff (see~\cite{SaffTotik}, Chapter~0). The
subset $\mathcal{M}^1(K)$ of all Borel probability measures on $K$~is
nonempty, convex and compact. The compactness follows from
Helly's Selection theorem.
% the
% well-known fact that $\mathcal{M}(K)$ is the dual space to the Banach
%space $(C(K),\|\cdot\|_\infty)$ of continuous functions on
% $K$. Hence $\mathcal{M}^1(K)$ is compact by Banach-Alaoglu's Theorem
%in the weak topology as it is the unit ball of
% $\mathcal{M}(K)$. Note here that in this case the weak topology and
%the topology of vague convergence coincide.
We will further
restrict to measures $\mu$ which are symmetric around $0$, that is,
$\mu(A)=\mu(-A)$ for all Borel sets $A$. It is easy to see
that
this subset still fulfills the assumptions of Schauder's fixed point
theorem.

Now we show that
since $h''(t)\geq-\alpha_Q$ and $h$ is bounded, the support of the
equilibrium measure to the external field $Q(t)+h_\mu(t)$ is included
in a compact set which can be chosen to be independent of
$\mu$. Indeed, by Theorem~\ref{thrm3}, the support of the equilibrium
measure for $Q(t)+h_\mu(t)$ is the smallest compact set~$K$
(w.r.t.
inclusion) of
positive capacity maximizing the functional
%
%e25 #&#
\begin{eqnarray}\label{49}
\quad K\mapsto F_{Q+h_\mu}(K)&=&\log\operatorname{cap}(K)-2\int Q(t) \,d
\omega_K(t)-2\int h_\mu(t)\,d\o_K(t)
\nonumber\\[-8pt]\\[-8pt]
&=&F_Q(K)-2\int h_\mu(t)\,d\o_K(t),\nonumber
\end{eqnarray}
in particular we have
%e26 #&#
%e27 #&#
\begin{eqnarray}\label{p44}
F_{Q+h_\mu}(\operatorname{supp}\mu_Q) &\geq& F_Q(
\operatorname{supp}\mu_Q)-2\|h\|_\infty\in\mathbb{R}
\nonumber\\[-10pt]\\[-10pt]
\eqntext{\mbox{since }\llvert h_\mu\rrvert\leq\|h\|_\infty.}
\end{eqnarray}
As $Q$ is convex and symmetric, $\operatorname{supp}\mu_Q$ is a
symmetric interval;
see Theorem~\ref{thrm3}. Because $h$ is twice differentiable,
$h'$ (and by assumption also $h$) are bounded on any compact set.
Hence, if we choose a
probability measure $\mu$ with compact
support, $h_\mu$ is two times differentiable and $(h_\mu)''=(h'')_\mu
$. By the condition $h''(t)\geq-\alpha_Q$, $Q(t)+h_\mu(t)$ is
convex for each compactly supported $\mu$.
Theorem~\ref{thrm3} implies that the support of the equilibrium
measure to $Q(t)+h_\mu(t)$ is a~symmetric interval, say
$[-l_\mu,l_\mu]$.
Using Lemma~\ref{lemma3}, we can rewrite (\ref{49}) for an arbitrary symmetric
interval $[-l,l]$ as
%
%e28 #&#
\begin{eqnarray}\label{p32}
F_{Q+h_\mu}\bigl([-l,l]\bigr)&=&\log(l/2) - 2\int_{-l}^l
Q(t)\frac{1}{\pi \sqrt{l^2-t^2}}\,dt
\nonumber\\[-8pt]\\[-8pt]
&&{}  -2\int_{-l}^l h_\mu(t)\frac{1}{\pi
\sqrt{l^2-t^2}}\,dt.\nonumber
\end{eqnarray}
Since $Q$ is strictly
convex and symmetric, we have $Q(t)\geq{\alpha_Q}t^2+C$ for some
$C\in\mathbb{R}$, and (\ref{p32}) implies (using that the variance of
$\o_{[-l,l]}$ is $l^2/2$) the inequality
%
%e29 #&#
\begin{equation}
F_{Q+h_\mu}\bigl([-l,l]\bigr)\leq\log(l/2) - \alpha_Q
l^2-C +2\|h\|_\infty,\label{110}
\end{equation}
which holds for any $\mu$. Comparing (\ref{p44}) and (\ref{110}), we
see that
%$F_{Q+h_\mu}([-l,l])$ is smaller than
%$F_Q(\operatorname{supp}\mu_Q)-2\|h\|_\infty$ if $l$ exceeds some $
%
\[
F_{Q+h_\mu}(\operatorname{supp}\mu_Q)>F_{Q+h_\mu}
\bigl([-l,l]\bigr)
\]
for all $l>L$, where $L>0$ does not depend on $\mu$. Hence such an
$[-l,l]$ cannot be the support $[-l_\mu,l_\mu]$ of the
equilibrium measure for $Q+h_\mu$. Hence $l_\mu\leq L$ for all compactly
supported $\mu$.

We have thus seen that $T_h$ maps the set $\mathcal{M}_s^1(K)$ of
symmetric probability measures supported in $K$ into itself, if
$K$ is chosen large enough. It remains to show continuity of
this map. Since we deal with a metric space, it is enough to show that
by $T_h$, converging sequences are mapped to converging
sequences. Let $(\mu_n)_n\subset
\mathcal{M}^1(K)$ be a sequence converging vaguely, or equivalently,
weakly to a probability measure
$\mu$. Denote
$T_h(\mu_n)=:\nu_n$. Define the sequence of external fields
$V_n(t):=Q(t)+h_{\mu_n}(t)$ which converges pointwise to
$V(t):=Q(t)+h_\mu(t)$. We may assume that this convergence is uniform:
by Theorem~\ref{thrm4}, the
equilibrium measure does not depend on values of the external field
outside of its support (from which we know a priori that it
lies in a certain compact set). Since $h'$ is bounded on this compact
set by some constant, say $C$, we also have $\llvert
h'_{\mu_n}\rrvert\leq C$. This implies that the sequence of
functions $(h_{\mu_n})_n$ is uniformly Lipschitz and hence equicontinuous.
It follows that the sequence $(V_n)_n$ is also equicontinuous. Since
their domain is a compact and $V_n$ converges pointwise,
the equicontinuity implies uniform convergence by the Arzela--Ascoli
theorem.

Since all $\nu_n$ are supported on the same compact set, it follows
that $(\nu_n)_n$ is tight and hence has a weakly converging
subsequence $(\nu_{n_m})_m$. We will prove that this limit measure,
say $\nu'$, is in fact $\nu=T_h(\mu)$, the measure belonging to
the external field $V$, and does not depend on the
particular subsequence. It follows that the sequence $(\nu_n)_n$
converges to $\nu$ weakly as weak convergence is
metrizable.

From the uniform convergence of $V_n$ toward $V$, it follows by Theorem
\ref{thrm5}(1) that
\[
U^{\nu_{n_m}}(s)= \int\log\llvert t-s\rrvert^{-1}\,d
\nu_{n_m}(t)
\]
converges uniformly (on $\mathbb{C}$) toward $U^{\nu}(s):= \int\log
\llvert t-s\rrvert^{-1}\,d\nu(t)$. On the other hand, by
Theorem~\ref{thrm5}(2)
we
have for almost all $s\in\mathbb{C}$
\[
\lim_{m\to\infty}U^{\nu_{n_m}}(s)=U^{\nu'}(s)= \int\log
\llvert t-s\rrvert^{-1}\,d\nu'(t).
\]
Hence $U^{\nu}(s)=U^{\nu'}(s)$ almost everywhere on $\mathbb{C}$.
Theorem~\ref{thrm5}(3) yields that
\mbox{$\nu=\nu'$}, implying that the
sequence $(\nu_n)_n$ converges weakly to $\nu$. As $T_h$ is a
continuous mapping, Schauder's fixed point theorem yields the
existence of a fixed point.
\end{pf}

%
%re3.2 #&#
\begin{remark}[(Uniqueness)]
So far we did not prove that this fixed point of $T_h$ is unique.
Uniqueness will follow for the class
of
ensembles from Theorem~\ref{300}. For those ensembles we will show
that the first correlation measure converges weakly to any
fixed point, which shows uniqueness.
\end{remark}
We proceed by decomposing the additional interaction term. Let $h$ be
as in Lemma~\ref{fixedpoint}.
Choose a fixed point $\mu_Q^h$ as in Lemma~\ref{fixedpoint}. We\vspace*{-1pt} will
stick to this measure from now on and write $\mu$ instead
of $\mu_Q^h$. We set using the notation (\ref{notationp})
\begin{eqnarray*}
\sum_{i<j}h(x_i-x_j)
&=& - \frac{N^2}{2}h_{\mu\mu}-\frac{N}{2}h(0)
+N\sum_{j=1}^N h_\mu(x_j)
\\
&&{} + \frac{1}{2} \Biggl(\sum_{i,j=1}^Nh(x_i-x_j)-
\bigl[h_\mu(x_i)+h_\mu(x_j)-h_{\mu\mu}
\bigr] \Biggr)
\\
&=& -\frac{N^2}{2}h_{\mu\mu}-\frac{N}{2}h(0)+N\sum
_{j=1}^N h_\mu(x_j)-\mathcal{U}(x),
\end{eqnarray*}
where
%e30 #&#
\begin{eqnarray}\label{51}
\mathcal{U}(x)&:=&-\frac{1}{2} \Biggl(\sum_{i,j=1}^Nh(x_i-x_j)-
\bigl[h_\mu(x_i)+h_\mu(x_j)-h_{\mu\mu}
\bigr] \Biggr).
\end{eqnarray}
Now we can rewrite $P_{N,Q}^h$ as
%
%e31 #&#
\begin{equation}
P_{N,Q}^h(x)=\frac{1}{Z_{N,V, \mathcal{U}}}\prod
_{1\leq i<j\leq
N}\llvert x_i-x_j\rrvert
^2e^{-N\sum_{j=1}^NV(x_j)+\mathcal
{U}(x)},\label{rewritten}
\end{equation}
where we defined the external
field
\[
V(t):=Q(t)+h_\mu(t)
\]
and absorbed the constant $\exp\{-(N^2/2)h_{\mu\mu}-(N/2)h(0)\}$
into the new normalizing constant
$Z_{N,V, \mathcal{U}}$. We will from now on work with this
representation of the density of $P_{N,Q}^h$. The proofs of Theorems
\ref{300}~and~\ref{312} rely on
comparison with the unitary invariant matrix ensemble
%
%e32 #&#
\begin{equation}
P_{N,V}(x)=\frac{1}{Z_{N,V}}\prod_{1\leq i<j\leq N}
\llvert x_i-x_j\rrvert^2e^{-N\sum_{j=1}^NV(x_j)}.
\label{eq2}
\end{equation}
We will show that in the large $N$ limit, the correlation measures in
the global scaling as well as correlation functions in the
local scaling, are the same for $P_{N,Q}^h$~and~$P_{N,V}$. In this
sense the quantity $\mathcal{U}$ will turn out to be
negligible.

%s4 #&#
\section{Concentration of measure inequalities}\label{globasym}

We will frequently use the following well-known concentration of
measure inequality (\cite{AGZ}, Section~4.4).\vadjust{\goodbreak}

%
%th4.1 #&#
\begin{thrm}\label{Concentration}
Let $Q$ be an external field on an interval $I=(a,b)$ (possibly
unbounded) with $Q''\geq c>0$ on~$I$. Then we have
% \begin{enumerate}
% \item
for any Lipschitz function $f$ on~$I$ and any $\varepsilon>0$
\[
P_{N,Q} \Biggl(\Biggl\llvert\sum_{j=1}^N
f(x_j)-\mathbb{E}_{N,Q}\sum_{j=1}^N
f(x_j)\Biggr\rrvert>\varepsilon\Biggr)\leq2\exp\biggl\{{-
\frac{c\varepsilon^2}{2\llvert f \rrvert
_{\mathcal{L}}2}} \biggr\}
\]
and
\[
\mathbb{E}_{N,Q}\exp\Biggl\{\varepsilon\Biggl(\sum
_{j=1}^N f(x_j)-\mathbb{E}_{N,Q}
\sum_{j=1}^N f(x_j) \Biggr)
\Biggr\}\leq\exp\biggl\{{\frac{\varepsilon^2\llvert f \rrvert
_{\mathcal{L}}^2}{2c}} \biggr\},
\]
where for any Lipschitz function $f$ we denote its Lipschitz constant
by $\llvert f \rrvert_{\mathcal{L}}$ (on~$I$).
\end{thrm}

%
%re4.2 #&#
\begin{remark}
In \cite{AGZ}, only the case $(a,b)=\mathbb{R}$ is stated. As the
proof for general $(a,b)$ is completely analogous, we do not give it
here.
\end{remark}

Theorem~\ref{Concentration} yields a concentration inequality for
linear statistics around their expectations. However, we rather
need
concentration
around their ``limiting expectations.'' It is well known (see, e.g.,
\cite{Johansson98}, Theorem 2.1) that for bounded and continuous
functions
%
%e33 #&#
\begin{equation}
\label{p1} \lim_{N\to\infty}\frac{1}{N}\mathbb{E}\sum
_{j=1}^Nf(x_j)=\int f(t)\,d
\mu_Q(t),
\end{equation}
where $\mu_Q$ denotes the equilibrium measure to $Q$. We need to
quantify the rates
of convergence in (\ref{p1}). The following is a special case of a
result in \cite{Shcherbina}; see also~\cite{KriecherbauerShcherbina}.

%
%pr4.3 #&#
\begin{prop}\label{KrSh}
Let $Q$ be a convex external field on $\mathbb{R}$ which is real
analytic in a neighborhood of $\operatorname{supp}(\mu_Q)$. Let $f$ be a
function whose third derivative is bounded on a neighborhood of
$\operatorname{supp}(\mu_Q)$. Then
\[
\Biggl\llvert\mathbb{E}_{N,Q}\sum_{j=1}^Nf(x_j)-N
\int f\,d\mu_Q\Biggr\rrvert\leq C \bigl(\|f\|_\infty+
\bigl\|f^{(3)}\bigr\|_\infty\bigr),
\]
where $C$ does not depend on $N$ or $f$, and $\|\cdot\|_\infty$
denotes the bound on the neighborhood of $\operatorname{supp}(\mu_Q)$.
\end{prop}
From Theorem~\ref{Concentration} and Proposition~\ref{KrSh} we
immediately get the following concentration inequality.

%
%co4.4 #&#
\begin{cor}\label{Concentration2}
Let $Q$ be a real analytic external field with $Q''\geq c>0$. Then
for any Lipschitz function $f$ whose third derivative is bounded on a\vadjust{\goodbreak}
neighborhood of $\operatorname{supp}(\mu_Q)$, we have for any
$\varepsilon>0$
\begin{eqnarray*}
&& \mathbb{E}_{N,Q}\exp\Biggl\{{\varepsilon\Biggl(\sum
_{j=1}^N f(x_j)-N\int f(t)\,d
\mu_Q(t) \Biggr)} \Biggr\}
\\
&&\qquad \leq\exp\biggl\{{\frac
{\varepsilon^2\llvert f \rrvert_{\mathcal{L}}^2}{2c}}+ \e C \bigl(\|
f\|_\infty+
\bigl\|f^{(3)}\bigr\|_\infty\bigr) \biggr\}.
\end{eqnarray*}
\end{cor}

% \begin{proof}
% The only argument we have to add to Theorem~\ref{Concentration} and
%Proposition~\ref{KrSh} is that we can change in Proposition
%~\ref{KrSh} from the ensemble $P_{N,Q}$ on $\R^N$ to $P_{N,Q|_I}$ on
%$I^N$. It is well-known (see e.g.
% \cite{PasturShcherbina},\cite{BorotGuionnet}) that changing the
%external field outside a small neighborhood of the equilibrium
% measure results in a change of the first correlation function of
%order $e^{-cN}$ for some $c>0$. This will follow from Lemma
%~\ref{lemma9} provided that $I$ is large enough.
% \end{proof}

%
%re4.5 #&#
\begin{remark}
Proposition~\ref{KrSh} and Corollary~\ref{Concentration2} remain true
up to an error of order $e^{-cN}$ if we replace $\mathbb{R}$ by an
interval $I$ which covers the domain of the equilibrium measure $\mu_Q$.
It is well known (see, e.g., \cite{PasturShcherbina,BorotGuionnet}) that changing the external field outside a small
neighborhood of the equilibrium
measure results in a change of the first correlation function of order
$e^{-cN}$ for some $c>0$. We will prove this in Lemma
\ref{lemma9} provided that $I$ is large enough.
\end{remark}

The next lemma gives, using Fourier techniques, a representation of the
bivariate statistic $\mathcal{U}$ in terms of certain linear
statistics. A similar idea is used in \cite{LytovaPastur08}.

%
%le4.6 #&#
\begin{lemma}\label{lemma4}
The following holds:
\[
\mathcal{U}(x)=-\frac{1}{2\sqrt{2\pi}}\int\bigl\llvert\cir{u}_N(t,x)
\bigr\rrvert^2 \hat{h}(t)\,dt,
\]
where
\begin{eqnarray*}
\cir{u}_N(t,x)&:=&\sum_{j=1}^N
\cos(tx_j)-N\int\cos(ts)\,d\mu(s)+\sqrt{-1}\sum
_{j=1}^N \sin(tx_j),
\\
\hat{h}(t)&:=&\frac{1}{\sqrt{2\pi}}\int_{\mathbb{R}}e^{-its}h(s)\,ds.
\end{eqnarray*}
\end{lemma}

\begin{pf}
Recall from (\ref{51}) that
\[
\mathcal{U}(x)=-\frac{1}{2} \Biggl(\sum_{i,j=1}^Nh(x_i-x_j)-
\bigl[h_\mu(x_i)+h_\mu(x_j)-h_{\mu\mu}
\bigr] \Biggr).
\]
Note that
\begin{eqnarray*}
\frac{1}{2}\sum_{j,k} h(x_j-x_k)&=&
\frac{1}{2\sqrt{2\pi}}\int\sum_{j,k}e^{i(x_j-x_k)t}
\hat{h}(t)\,dt\\
&=&\frac{1}{2\sqrt{2\pi}}\int\bigl\llvert u_N(t,x)\bigr
\rrvert^2 \hat{h}(t)\,dt
\end{eqnarray*}
with $u_N(t,x):=\sum_{j=1}^Ne^{itx_j}$. Writing $\cir
{u}_N(t,x):=u_N(t,x)-N\int e^{its}\,d\mu(s)$,
it is not hard to check that
%
%e34 #&#
\begin{equation}
\mathcal{U}(x)=-\frac{1}{2\sqrt{2\pi}}\int\bigl\llvert\cir{u}_N(t,x)
\bigr\rrvert^2 \hat{h}(t)\,dt.\label{23}
\end{equation}\upqed
\end{pf}
Note that we can write
\[
\mathbb{E}_{N,Q}^hf(x)=(Z_{N,V}/{Z_{N,V, \mathcal{U}}})
\mathbb{E}_{N,V}f(x)e^{\mathcal{U}(x)}.
\]
With the help of representation (\ref{23}), we shall bound this ratio
of normalizing constants.

%
%pr4.7 #&#
\begin{prop}\label{Fourier} If the constant $\alpha_Q$ is large
enough, then there exist constants $C_1, C_2>0$ such that for
all $N$
\[
0<C_1\leq Z_{N,V, \mathcal{U}}/Z_{N,V}=\mathbb{E}_{N,V}
\exp\bigl\{ {\mathcal{U}(x)} \bigr\}\leq C_2.
\]
\end{prop}
\begin{pf}
We start with proving the lower bound.
By Jensen's
inequality we see
\[
\mathbb{E}_{N,V}\exp\bigl\{{\mathcal{U}(x)} \bigr\}\geq\exp\bigl\{ {
\mathbb{E}_{N,V} \mathcal{U}(x)} \bigr\}.
\]
Using Lemma~\ref{lemma4} we show that the expectation of $\mathcal
{U}$ is bounded in $N$.
Fubini's theorem gives
\begin{eqnarray*}
-\mathbb{E}_{N,V}\mathcal{U}(x)&=&\frac{1}{2\sqrt{2\pi}}\int
\mathbb{E}_{N,V}\bigl\llvert\cir{u}_N(t,x)\bigr\rrvert
^2 \hat{h}(t)\,dt
\\
&=&\frac{1}{2\sqrt{2\pi}}\int\Biggl(\mathbb{E}_{N,V}\Biggl\llvert\sum
_{j=1}^N \cos(tx_j)-N\int
\cos(ts)\,d\mu(s) \Biggr\rrvert^2
\\
&&\hspace*{133pt}{}+\mathbb{E}_{N,V}\Biggl\llvert\sum_{j=1}^N
\sin(tx_j) \Biggr\rrvert^2 \Biggr)\hat{h}(t)\,dt.
\end{eqnarray*}
By\vspace*{1pt} Corollary~\ref{Concentration2}, the terms in the parentheses are
bounded by a polynomial function in $t$, as
$\llvert\cos(t\cdot) \rrvert_{\mathcal{L}},\llvert
\sin(t\cdot) \rrvert_{\mathcal{L}}\leq t$ and $\|{\cos
(t\cdot)}^{(3)}\|_\infty,\|{\sin(t\cdot)}^{(3)}\|_\infty\leq
Ct^3$. Hence,
$\hat{h}$ being a Schwartz function, we have $\mathbb
{E}_{N,V}\mathcal
{U}(x)\geq-C'$ for some
$C'>0$.
Thus the lower bound follows choosing $C_1:=\exp(-C')$.

For the upper bound we will again use the representation of Lemma~\ref{lemma4}.
Recall that since $h$
is even, $\hat{h}$ is real-valued. Define $\hat{h}_+(y):=\max
\{0,\hat{h}(y)\}$ and
$\hat{h}_-(y):=\max\{0,-\hat{h}(y)\}$ such that $\hat{h}=\hat{h}_+-\hat{h}_-$. For $\hat{h}_-=0$, which\vspace*{2pt} corresponds to the case of a positive
definite $h$, there is nothing to prove, so assume that $\hat{h}_-\neq
0$.

Introducing $H_{-}:= (\hat{h}_{-} )^{1/2}\ge0$, we
obtain by
Jensen's inequality and Tonelli's theorem
%
%e35 #&#
\begin{eqnarray}\label{p14}
&&\mathbb{E}_{N,V}\exp\biggl\{-(2\sqrt{2\pi})^{-1}\int
\hat{h}(t)\bigl\llvert\cir{u}_N(t,x)\bigr\rrvert^2
\,dt \biggr\}\hspace*{-10pt}
\nonumber
\\
&&\qquad \leq\mathbb{E}_{N,V}\exp\biggl\{ (2\sqrt{2\pi})^{-1}\int
H_{-}(t)^2 \bigl\llvert\cir{u}_N(t,x)\bigr
\rrvert^2 \,dt \biggr\}\hspace*{-10pt}
\nonumber\\[-8pt]\\[-8pt]
&&\qquad=\mathbb{E}_{N,V}\exp\biggl\{(2\sqrt{2\pi})^{-1} \|H_- \|
_{L^1}
\int\bigl(H_-(t) / \|H_- \|_{L^1} \bigr) H_-(t)\bigl
\llvert\cir{u}_N(t,x) \bigr\rrvert^2 \,dt \biggr\}\hspace*{-10pt}
\nonumber
\\
&&\qquad \leq\int\bigl(H_-(t) / \|H_- \|_{L^1} \bigr) \mathbb{E}_{N,V}
\exp\bigl\{{(2\sqrt{2\pi})^{-1} \|H_- \|_{L^1}H_-(t)\bigl
\llvert\cir{u}_N(t,x)\bigr\rrvert^2} \bigr\}
\,dt.\nonumber\hspace*{-10pt}
\end{eqnarray}
Abbreviating $K_h:=(2\sqrt{2\pi})^{-1} \|H_- \|_{L^1}$ and
using the Cauchy--Schwarz inequality and representation
(\ref{23}), we find
%
%e36 #&#
%e37 #&#
%e38 #&#
\begin{eqnarray}
&&\mathbb{E}_{N,V} \exp\bigl\{K_h H_-(t)\bigl\llvert
\cir{u}_N(t,x)\bigr\rrvert^2 \bigr\} \label{p11}
\\
&&\qquad \leq\mathbb{E}_{N,V}^{1/2} \exp\Biggl\{ 2 K_h
H_-(t) \Biggl|\sum_{j=1}^N
\cos(tx_j)-N\int\cos(ts)\,d\mu(s) \Biggr|^2 \Biggr
\}\label{p12}
\\
&&\quad\qquad{} \times\mathbb{E}_{N,V}^{1/2} \exp\Biggl\{2
K_h H_-(t) \Biggl| \sum_{j=1}^N
\sin(tx_j) \Biggr|^2 \Biggr\}.\label{p13}
\end{eqnarray}
Since by Corollary~\ref{Concentration2} the distributions of $\sum
_{j=1}^N \cos(tx_j)-N\int
\cos(ts)\,d\mu(s)$ and $ \sum_{j=1}^N \sin(tx_j)$ are sub-Gaussian,
we obtain, for example, for the first term for any
$\e>0$,
%
%e39 #&#
\begin{eqnarray}\label{p10}
&&\mathbb{E}_{N,V} \exp\Biggl\{\e\cdot\sqrt{2 K_h H_-(t)}
\Biggl(\sum_{j=1}^N \cos
(tx_j)-N\int\cos(ts)\,d\mu(s) \Biggr) \Biggr\}
\nonumber\\[-8pt]\\[-8pt]
&&\qquad \leq\exp\bigl\{\e^2\cdot2 K_h H_-(t)t^2 (2
\alpha_V)^{-1}+\e\sqrt{2K_h H_-(t)}C
\bigl(1+t^3\bigr) \bigr\}, \nonumber
\end{eqnarray}
where $\alpha_V:=\min_{t}V''(t)>0$, $C$ does not depend on $t$ or
$N$. For $\alpha_Q$ large enough
(hence $\alpha_V$ large enough), we have $2 K_h H_-(t)t^2(2\alpha
_V)^{-1}<1/4$ for all $t$. Since $H_{-}(t)=\hat{h}^{1/2}_-(t)$ is
decaying rapidly,
$\sqrt{2K_h H_-(t)}C(1+t^3)$ is bounded in $t$. Summarizing, if
$\alpha_Q$ is large enough, we can bound (\ref{p10}) by
\[
\exp\bigl\{ c\e^2+\e C \bigr\}
\]
with $0<c<1/4$ and $c,C$ do not depend on $N$ or $t$. We conclude that
(\ref{p12})~and~(\ref{p13}) and hence (\ref{p11}) are bounded in
$N$. Finally, since $\hat{h}$ is a Schwartz function, it follows
from
(\ref{p14}) that
\[
\mathbb{E}_{N,V}\exp\biggl\{{-\int\hat{h}(t)\bigl\llvert\cir
{u}_N(t,x)\bigr\rrvert^2\,dt} \biggr\}\leq C
\]
for some constant $C>0$ independent of $N$. This proves the upper bound
and hence the proposition.
\end{pf}

%
%re4.8 #&#
\begin{remark}
The proof of Proposition~\ref{Fourier} actually shows that for each
$\lambda>0$ there is a threshold $\alpha^h(\lambda)>0$ and constants $C_1,
C_2$
(depending on $\lambda$ and $\alpha^h$) such that
\[
0<C_1<\mathbb{E}_{N,V}\exp\bigl\{\lambda\mathcal{U}(x)
\bigr\}\leq C_2\qquad\mbox{if } \alpha_Q\geq
\alpha^h(\lambda).
\]
\end{remark}
%
%s5 #&#
\section{Bounding the first correlation function}\label{sec5}
This section deals with properties of the first correlation function.
We give information on its decay and dependence on
additional external fields of lower order.

First of all, we need to introduce some notation from \cite{Johansson98}:
%
%e40 #&#
%e41 #&#
%e42 #&#
\begin{eqnarray}
K_{N,Q}(x)&:=&\sum_{1\leq i\neq j\leq N}k_Q(x_i,x_j),
\nonumber\\[-8pt]\label{68} \\[-8pt]
k_Q(t,s)&:=& \log\llvert t-s\rrvert^{-1}+
\frac{1}{2}Q(t)+\frac{1}{2}Q(s),\nonumber
\\
F_Q&:=&I_{Q}(\mu),\qquad \psi_Q(t):=Q(t)-\log
\bigl(t^2+1\bigr)
\nonumber\\[-8pt]\label{70}\\[-8pt]
\eqntext{\mbox{where }I_Q(\mu)\mbox{ is defined in (\ref{120}).}}
\end{eqnarray}
From the simple inequality $\llvert t-s\rrvert\leq\sqrt{t^2+1}\sqrt
{s^2+1}$ we conclude $\log\llvert t-s\rrvert
^{-1}\geq
-\frac{1}{2}\log(t^2+1)(s^2+1)$ and hence
%
%e43 #&#
\begin{equation}
\label{52} k_Q(t,s)\geq(1/2)\psi_Q(t)+(1/2)
\psi_Q(s).
\end{equation}
We also note that since $Q$ is an external field, there is a constant
$c_Q$ such that
%
%e44 #&#
\begin{equation}
\label{62} \psi_Q(t)\geq c_Q.
\end{equation}
We define a generalized unitary invariant ensemble on $\mathbb{R}^N$
(or some compact $[a,b]^N$) via
%
%e45 #&#
\begin{equation}
\label{61} P_{N,Q,f}^M(x):=\frac{1}{Z_{N,Q,f}^M}\prod
_{1\leq i<j\leq N}\llvert x_i-x_j
\rrvert^2e^{-M\sum_{j=1}^N Q(x_j)+\sum_{j=1}^N f(x_j)},
\end{equation}
where $N,M\in\mathbb{N}$ and $f$ is a continuous function with
$\llvert f(t)\rrvert\leq Q(t)$ for $t$ large enough.
Usually we have $M=N$ or
$M=N-1$. If $M=N$, we will write $P_{N,Q,f}$ instead of $P_{N,Q,f}^M$.
If $f=0$, we write $P_{N,Q}^M$. The following result is due
to Johansson.

%
%pr5.1 #&#
\begin{prop}\label{LDP}
Let
\[
A_{N,\varepsilon}:= \biggl\{ x\in\mathbb{R}^N\dvtx
\frac
{1}{N^2}K_{N,Q}(x)\leq F_Q+\varepsilon\biggr\}.
\]
Then there is some constant $C$ such that, if $\lim_{N\to\infty
}N/M_N\to1$,
\[
P_{N,Q}^{M_N}\bigl(\mathbb{R}^N\setminus
A_{N,\varepsilon+a}\bigr)\leq Ce^{-aN^2}\qquad\mbox{for all }N\geq
N_0(\varepsilon)\mbox{ and all } a\geq0.
\]
\end{prop}

\begin{pf}
See \cite{Johansson98}, Lemma 4.2.
\end{pf}

We\vspace*{1pt} now deal with the decay of $\rho_{N,Q}^1$. The
following lemma can be found in several papers
including \cite{Johansson98,PasturShcherbina}. We follow \cite{Johansson98}.
%
%le5.2 #&#
\begin{lemma}\label{lemma5}
Let $Q$ be a continuous function satisfying $Q(t)\geq(1+\delta)\log
(1+t^2)$ for some $\delta>0$ and all $t$ large enough.
Then there is a constant $C>0$ such that for all $t$,
\[
\rho_{N,Q}^1(t)\leq e^{CN}e^{-N [Q(t)-\log(1+t^2) ]}.
\]
\end{lemma}

\begin{pf}
We will from now on drop the subscript $Q$, defining
\[
P_{N}^M(x):=\frac{1}{Z_{N}^M}\prod
_{1\leq i<j\leq N}\llvert x_i-x_j\rrvert
^2e^{-M\sum_{j=1}^N Q(x_j)}
\]
and abbreviating $\rho^1_N:=\rho^1_{N,Q}$, we compute
%
%e46 #&#
\begin{eqnarray}\label{72}
\rho^1_N(t)&=&\frac{Z_{N-1}^N}{Z_{N}^N}\mathbb{E}_{N}^{N-1}
\Biggl(\prod_{j=1}^{N-1}
(x_j-t )^2 \Biggr)e^{-NQ(t)},
\nonumber\\[-9pt]\\[-9pt]
\frac{Z_{N}^{N}}{Z_{N-1}^N}&=&\mathbb{E}_{N-1}^N \biggl(\int
e^{2\sum_{j=1}^{N-1}\log\llvert x_j-t\rrvert-NQ(t)}\,dt \biggr).\nonumber
\end{eqnarray}
Since adding a constant to $Q$ does not change the ensemble, we will
assume that $Q\geq0$, which corresponds to
considering the potential $Q+C_Q$, where $C_Q$ denotes a lower bound of $Q$.
Setting $Z:=\int e^{-Q(t)}\,dt$ we get by Jensen's inequality
\begin{eqnarray*}
&& Z\frac{1}{Z}\int\exp\Biggl\{{2\sum_{j=1}^{N-1}
\log\llvert x_j-t\rrvert-NQ(t)} \Biggr\}\,dt
\\
&&\qquad \geq Z\exp\Biggl\{{\frac{1}{Z}\int\Biggl(2\sum
_{j=1}^{N-1}\log\llvert x_j-t\rrvert
-(N-1)Q(t) \Biggr)e^{-Q(t)}\,dt} \Biggr\}.
\end{eqnarray*}
Since $Q\geq0$, we get
\[
\int\log\llvert t-x_j\rrvert e^{-Q(t)}\,dt\geq\int
_{x_j-1}^{x_j+1}\log\llvert t-x_j\rrvert
\,dt=-2.
\]
Summarizing we see that
%
%e47 #&#
\begin{equation}
{Z_{N}^{N}}/{Z_{N-1}^N}\geq Z\exp
\{{-CN}\} \qquad\mbox{for some constant } C>0.\label{73}
\end{equation}
Using the inequality $(x_j-t)^2\leq(1+x_j^2)(1+t^2)$ gives
%
%e48 #&#
\begin{equation}
\mathbb{E}_{N}^{N-1} \Biggl(\prod
_{j=1}^{N-1} (x_j-t )^2
\Biggr)\leq\bigl(1+t^2\bigr)^{N}\mathbb{E}_{N}^{N-1}
\Biggl(\prod_{j=1}^{N-1} \bigl(1+
x_j^2 \bigr) \Biggr).\label{74}
\end{equation}
As before, we can assume (otherwise we add a constant) that $Q$
satisfies $Q(t)\geq(1+\delta)\log(1+t^2)$ for all $t$
and some $\delta>0$.
Using notation (\ref{68})--(\ref{70}) and inequality~(\ref{52}), this
condition yields
\[
K_{N-1,Q}(x)\geq\delta(N-1)\sum_{j=1}^{N-1}
\log(1+x_j)^2.
\]
Proposition~\ref{LDP} shows that for $A$ large enough we
have
%
%e49 #&#
\begin{eqnarray} \label{71}
&&P_{N-1,Q}^N \Biggl(\sum_{j=1}^{N-1}
\log(1+x_j)^2\geq AN \Biggr)
\nonumber\\[-8pt]\\[-8pt]
&&\qquad \leq P_{N-1,Q}^N \bigl( K_{N-1,Q}(x) \geq\delta
A(N-1)N \bigr) \leq e^{-cAN^2}\nonumber
\end{eqnarray}
for some constant $c>0$. From this we conclude that for $A$ large enough
\[
\mathbb{E}_{N}^{N-1} \Biggl(\prod
_{j=1}^{N-1} \bigl(1+ x_j^2
\bigr) \Biggr)\leq e^{AN}+\mathbb{E}_{N}^{N-1}
\Biggl(\prod_{j=1}^{N-1} \bigl(1+
x_j^2 \bigr)\mathbh{1}_{\prod_{j=1}^{N-1} (1+ x_j^2 )\geq
e^{AN}} \Biggr).
\]
Equation (\ref{71}) gives that
\[
P_{N-1,Q}^N \Biggl(\sum_{j=1}^{N-1}
\log(1+x_j)^2 -AN\geq\llvert y\rrvert\Biggr)\leq\exp
\bigl\{{-cAN^2-c\llvert y\rrvert N}\bigr\}.
\]
From\vspace*{-2pt} this bound it is easy to see that
$\mathbb{E}_{N}^{N-1}
(\prod_{j=1}^{N-1}(1+ x_j^2)\mathbh{1}_{\prod_{j=1}^{N-1} (1+
x_j^2 )\geq
\exp\{{AN}\}} )$ is of order $\exp\{{-CN^2}\}$ for some $C>0$.
Hence we have
%
%e50 #&#
\begin{equation}
\label{75} \mathbb{E}_{N}^{N-1} \Biggl(\prod
_{j=1}^{N-1} \bigl(1+ x_j^2
\bigr) \Biggr)\leq\exp\{{cAN}\}\qquad\mbox{for some }c.
\end{equation}
In view of (\ref{72}) we find combining (\ref{73}), (\ref{74}) and
(\ref{75}),
\[
\rho_{N,Q}^1(t)\leq\exp\{{CN}\}\exp\bigl\{{-N \bigl[Q(t)-
\log\bigl(1+t^2\bigr) \bigr]}\bigr\}.
\]\upqed
\end{pf}
From the previous lemma we easily deduce the following important
corollary; cf. \cite{Johansson98,PasturShcherbina,Deift98}.

%
%co5.3 #&#
\begin{cor}\label{cor1}
Let $Q$ be as in Lemma~\ref{lemma5}. Then there are $L,C>0$ such that
for all $t$ with $t>L$, we have
\[
\rho^1_N(t)\leq\exp\bigl\{{-CNQ(t)}\bigr\}.
\]
\end{cor}

We finish the section with a useful bound on the first correlation
function $\rho^1_{N,Q,f}$ of the unitary invariant ensemble
$P_{N,Q,f}$; see (\ref{61}).
%
%le5.4 #&#
\begin{lemma}\label{cor2}
Let $f$ be bounded. Then we have
\[
\rho^1_{N,Q,f}(t)\leq\rho^1_{N,Q}(t)
e^{2\|f\|_\infty}.
\]
\end{lemma}
\begin{pf}
We use the identity
%
%e51 #&#
\begin{equation}
\rho^1_{N,Q,f}(t)= \frac{e^{-NQ+f}}{N\lambda_N(e^{-NQ+f},t)},\label{p20}
\end{equation}
where $\lambda_N(e^{-NQ+f},\cdot)$ is the so-called $N$th Christoffel
function to the weight $e^{-NQ+f}$ (see \cite{Totik} for
references and more information on Christoffel functions)
%
%e52 #&#
\begin{equation}
\label{p21} \lambda_N(W,t):=\inf_{P_{N-1}(t)=1}\int\bigl
\llvert P_{N-1}(s)\bigr\rrvert^2W(s)\,ds,
\end{equation}
where the infimum is taken over all polynomials $P_{N-1}$ of at most
degree \mbox{$N-1$} with the property that $P_{N-1}(t)=1$ and $W$
denotes a weight function on $\mathbb{R}$. It~is obvious from (\ref
{p21}) that $\lambda_N(W_1,\cdot)\leq\lambda_N(W_2,\cdot)$ if
$W_1\leq W_2$. Then
the lemma follows easily by $ e^{-NQ-\|f\|_\infty}\leq e^{-NQ+f}\leq
e^{-NQ+\|f\|_\infty}$.
\end{pf}
%
%s6 #&#
\section{Proofs of Theorems \texorpdfstring{\protect\ref{300}}{1.1}
and \texorpdfstring{\protect\ref{312}}{1.2}}\label{sec6}
We first cite a general result by Levin and Lubinsky (\cite{LevinLubinsky08}, Theorem 1.1) about bulk universality
for unitary invariant ensembles. Recall the definition of $\rho
^k_{N,Q,f}$ following (\ref{61}).
%
%th6.1 #&#
\begin{thrm}\label{thrmLubinsky}
Let $Q$ be a continuous external field on the set $\Sigma\subset
\mathbb{R}$, which is assumed to consist of at most finitely many
intervals.
Let $f$ be a bounded continuous function on $\Sigma$. Let $K_N$ denote
the kernel
\[
K_N(t,s)=\sum_{j=0}^{N-1}
\psi^{\mathrm{(N)}}_j(t)\psi^{\mathrm{(N)}}_j(s),
\]
where $ (\psi^{\mathrm{(N)}}_j )_j$ are the
orthonormal
functions to the weight $e^{-NQ(t)+f(t)}$. Let $J$~be a~closed
interval lying inside the support of $\mu_Q$. Assume that $\mu_Q$ is
absolutely continuous in a neighborhood of $J$ and that $Q'$
and
the
density $\mu_Q$ are continuous in that neighborhood, while $\mu_Q>0$
there. Then uniformly for $a\in J$ and $t,s$ in compacts of the
real line, we have
%
%e53 #&#
\begin{equation}
\label{121} \quad\lim_{N\to\infty}\frac{K_N (a+(t/(K_N(a,a))),
a+(s/(K_N(a,a)) )}{K_N (a,a )}=
\frac{\sin(\pi(t-s) )}{\pi(t-s)}.
\end{equation}
\end{thrm}
We use a notion of bulk universality which slightly differs from (\ref
{121}); namely we scale by the limiting density $\mu_Q$
instead
of
using the $N$-particle density. The following obvious corollary is a
translation of Theorem~\ref{thrmLubinsky} into this setup.
%
%co6.2 #&#
\begin{cor}\label{cor3}
Let $Q$, $f$ and $\mu_Q$ be as in Theorem~\ref{thrmLubinsky}. Then
bulk universality as defined in (\ref{p39}) holds for the
unitary invariant ensemble
$P_{N,Q,f}$.
\end{cor}
\begin{pf}
The corollary follows from the well-known determinantal relations for
unitary invariant ensembles, the local uniformness of the
limit (\ref{121}) in
$t,s$ and the fact that by \cite{Totik}, Theorem 1.2, we have uniformly
in compact proper subsets of $\operatorname{supp}\mu_Q$
\[
\lim_{N\to\infty}\frac{1}{N}K_N(a,a)=\lim
_{N\to\infty}\rho_{N,Q,f}^1(a)=
\mu_Q(a).
\]\upqed
\end{pf}

We will prove Theorems~\ref{300}~and~\ref{312} together by comparing
the correlation functions of the ensembles $P_{N,Q}^h$
[see (\ref{rewritten})]
and $P_{N,V}$; see (\ref{eq2}). We start with $\rho^k_{N,V}$, the
$k$th correlation function of $P_{N,V}$. We obtain
$\rho^k_{N,V} (
a+\frac{t_1}{N\mu(a)},\ldots,a+\frac{t_k}{N\mu(a)} )$ as
$k$-marginal, integrating the density
\[
P_{N,V} \biggl(a+\frac{t_1}{N\mu(a)},\ldots,a+\frac{t_k}{N\mu
(a)},x_{k+1},\ldots,x_N \biggr)
\]
over $x_{k+1},\ldots,x_N$. We have $k$ fixed eigenvalues at positions
$a+\frac{t_1}{N\mu(a)},\ldots,a+\frac{t_k}{N\mu(a)}$
and $N-k$ random eigenvalues. We\vspace*{1pt} first rewrite $\rho^k_{N,V}$ in terms
of these $N-k$ random eigenvalues as follows:
%
%e54 #&#
%e55 #&#
\begin{eqnarray}
&&\rho^k_{N,V} \biggl(a+\frac{t_1}{N\mu(a)},\ldots,a+
\frac{t_k}{N\mu
(a)} \biggr)
\nonumber
\\
&&\qquad =\int_{\mathbb{R}^{N-k}}\frac{1}{Z_{N,V}}\exp\Biggl\{-N\sum
_{j=k+1}^NV(x_j)+2\sum_{i<j; i,j>k} \log\llvert x_j-x_i\rrvert
\Biggr\}
\nonumber
\\
&&\hspace*{24pt}\qquad\quad{} \times\exp\Biggl\{-N\sum_{j=1}^k V
\biggl(a+\frac{t_j}{N\mu(a)} \biggr)+2\sum_{i<j; i,j\leq
k}\log
\biggl\llvert\frac{t_i-t_j}{N\mu(a)}\biggr\rrvert\Biggr\}\label
{31}
\\
&&\hspace*{24pt}\quad\qquad{}\times\exp\biggl\{2\sum_{i\leq k, j>k}\log\biggl\llvert a+
\frac
{t_i}{N\mu(a)}-x_j\biggr\rrvert\biggr\}\,dx_{k+1}\cdots dx_N
\nonumber
\\
\label{39}&&\qquad  = F(a,t)\frac{Z_{N-k,V}^N}{Z_{N,V}}\mathbb{E}_{N-k,V}^N\exp\biggl\{
{2\sum_{i\leq k, j>k}\log\biggl\llvert a+
\frac{t_i}{N\mu(a)}-x_j\biggr\rrvert} \biggr\},
\end{eqnarray}
where
%e56 #&#
\begin{equation}\label{p23}
F(a,t):=\exp\Biggl\{{-N\sum_{j=1}^kV
\biggl(a+\frac{t_j}{N\mu(a)} \biggr)+2\sum_{i<j; i,j\leq
k}\log
\biggl\llvert\frac{t_i-t_j}{N\mu(a)}\biggr\rrvert} \Biggr\}
\end{equation}
is the factor (\ref{31}), which depends only on the fixed particles,
and
\[
P_{N-k,V}^N(x_{k+1},\ldots,x_N):=
\frac{1}{Z_{N-k,V}^N}\prod_{k+1\leq
i<j\leq N}\llvert
x_i-x_j\rrvert^2e^{-N\sum_{j=k+1}^N V(x_j)}.
\]
As before, the subscript $N-k$ indicates that $P_{N-k,V}^N$ is a
probability measure in $N-k$ variables, whereas the superscript $N$
indicates that the factor in front of the external field term $\sum
_{j=k+1}^N V(x_j)$ of $P_{N-k,V}^N$ is $N$ and not $N-k$. We keep
the labeling $x_{k+1},\ldots,x_N$.
Setting
%
%e57 #&#
\begin{eqnarray}\label{eq3}
&& L_{N-k,V}^N(a,t,x)
\nonumber\\[-4pt]\\[-16pt]
&&\qquad :=2\sum_{i\leq k, j>k}
\log\biggl\llvert a+\frac{t_i}{N\mu(a)}-x_j\biggr\rrvert+\log
\biggl[F(a,t)\frac
{Z_{N-k,V}^N}{Z_{N,V}} \biggr],\nonumber
\end{eqnarray}
we get from (\ref{39}) the equality
%
%e58 #&#
\begin{equation}
\quad \rho^k_{N,V} \biggl(a+\frac{t_1}{N\mu(a)},\ldots,a+
\frac{t_k}{N\mu
(a)} \biggr)=\mathbb{E}_{N-k,V}^N\exp\bigl
\{{L_{N-k,V}^N(a,t,x)} \bigr\} \label{56}.
\end{equation}
Similar to (\ref{39}), we see that the $k$th correlation function
$\rho_{N,Q}^{h,k}$ of $P_{N,Q}^h$ at
$a+\frac{t_1}{N\mu(a)},\ldots,a+\frac{t_k}{N\mu(a)}$ can be written as
%
%e59 #&#
\begin{equation}
\frac{1}{\mathbb{E}_{N,V}\exp\{{\mathcal{U}(x)} \}}\mathbb
{E}_{N-k,V}^N\exp\bigl\{{
\mathcal{U}(t,x)+L_{N-k,V}^N(a,t,x)} \bigr\},\label{40}
\end{equation}
where we abbreviated $\mathcal{U}(a+\frac{t_1}{N\mu(a)},\ldots,a+\frac{t_k}{N\mu(a)},x_{k+1},\ldots,x_N)$ by $\mathcal{U}(t,x)$.

In the following we shall abbreviate $(t_1,\ldots,t_k,x_{k+1},\ldots,x_N)$ by $(t,x)$, and by
$(t,x)_j$ we will denote the $j$th component of the
vector $(t,x)$. Furthermore, for the sake of brevity, we set
%
%e60 #&#
\begin{equation}
R_a:=L_{N-k,V}^N(a,t,x)\quad\mbox{and}\quad
R:=L_{N-k,V}^N\bigl(0,N\mu(0)t,x\bigr)\label{p53}.
\end{equation}
Note that $R$ arises in the global scaling, whereas $R_a$ appears in
the local scaling.
It will later turn out to be convenient that
all the $x_j$'s lie in a compact set. To this end we formulate the following
truncation lemma. This procedure is well known for invariant ensembles;
see, for instance, \cite{Johansson98} or
\cite{BoutetdeMonvelPasturShcherbina}.

%le6.3 #&#
\begin{lemma}\label{lemma9}
For $\alpha_Q$ large enough, the following holds: for each $k$ there
are $L,C>0$ such that
for all $N$ and
for all $t_1,\ldots,t_k$
\[
\biggl\llvert\rho_{N,Q}^{h,k}(t_1,\ldots,t_k)-\frac{1}{\mathbb
{E}_{N,V,L}\exp\{{\mathcal{U}(x)} \}}\mathbb{E}_{N-k,V,L}^N
\exp\bigl\{{\mathcal{U}(t,x)+R_L} \bigr\}\biggr\rrvert\leq
e^{-CN},
\]
where $\mathbb{E}_{N,V,L}^M$ denotes expectation w.r.t. the ensemble
$P_{N,V,L}^M$ obtained by normalizing the ensemble $P_{N,V}^M$
restricted to $[-L,L]^N$ and $R_L$ is the analog of $R$ in which all
integrations over $\mathbb{R}$ have been replaced by integrations over
$[-L,L]$.
Furthermore, for any external field $Q$ on $\mathbb{R}$, the following
holds: for each $k$ there are $L',C>0$ such that
for all $N$ and all $t_1,\ldots,t_k$
\[
\bigl\llvert\rho_{N,Q}^{k}(t_1,\ldots,t_k)-\rho_{N,Q,L'}^{k}(t_1,\ldots,t_k)\bigr\rrvert\leq e^{-C'N},
\]
where\vspace*{-1pt} $\rho_{N,Q,L'}^k$ is the $k$th correlation function of the
ensemble $P_{N,Q,L'}$ obtained by normalizing the ensemble
$P_{N,Q}$ restricted to $[-L',L']^N$.
\end{lemma}
\begin{pf}
We will use representation (\ref{40}) and show that the restriction of
integrals to $[-L,L]^N\subset\mathbb{R}^N$, respectively,
$[-L,L]^{N-k}\subset\mathbb{R}^{N-k}$ results in an asymptotically
negligible error. For
$\mathbb{E}_{N,V}e^{\mathcal{U}}$ we use H\"older's inequality to estimate
\begin{eqnarray*}
&&\mathbb{E}_{N,V} \bigl(\exp\bigl\{\mathcal{U}(x)\bigr\}
\mathbh{1}_{
([-L,L]^N )^c}(x) \bigr)
\\
&&\qquad \leq\bigl(\mathbb{E}_{N,V} \exp\bigl\{(1+\e)\mathcal{U}(x)\bigr\}
\bigr)^{1/(1+\e)} \bigl( P_{N,V} \bigl( \bigl([-L,L]^N
\bigr)^c \bigr) \bigr)^{1/\e'},
\end{eqnarray*}
where $1/(1+\e)+1/\e'=1$ and $\e>0$ is fixed. Now $\mathbb{E}_{N,V}
e^{(1+\e)\mathcal{U}(x)}$ is uniformly bounded in $N$ by Proposition~\ref{Fourier}
provided that
$\alpha_Q$ is large enough. Furthermore, by Corollary~\ref{cor1} we
get for the $L$ defined there
%
%e61 #&#
\begin{eqnarray}\label{p24}
P_{N,V} \bigl( \bigl([-L,L]^N \bigr)^c \bigr)
&\leq& N\int_{\llvert
t\rrvert>L}\rho^1_{N,V}(t)\,dt
\nonumber\\[-8pt]\\[-8pt]
&\leq& N \int_{\llvert
t\rrvert>L}e^{-CNV(t)}\,dt\leq e^{-C'N}\nonumber
\end{eqnarray}
for some $C'>0$. In fact, $C'$ can be chosen arbitrarily large by
increasing $L$. We conclude that
\[
\mathbb{E}_{N,V} \bigl(\exp\bigl\{\mathcal{U}(x)\bigr\}
\mathbh{1}_{
([-L,L]^N )^c}(x) \bigr)\leq\exp\bigl\{-C''N
\bigr\}
\]
for some $C''>0$, if $L$ is large enough. It follows by (\ref{p24}) as
well that the exchange of the normalizing constants
$Z_{N,V}$ and
$Z_{N-k,V}^N$ by their counterparts $Z_{N,V,L}$, and $Z_{N-k,V,L}^N$
and hence also the exchange of $R$ by $R_L$ is asymptotically
negligible.

In\vspace*{-1pt} order to bound $\mathbb{E}_{N-k,V}^N (\exp\{{\mathcal
{U}(t,x)+R} \}\mathbh{1}_{ ([-L,L]^N )^c}(x) )$,
first use H\"ol\-der's
inequality as above. It remains to estimate
$\mathbb
{E}_{N-k,V}^N\exp\{(1+\e)\mathcal{U}(t,x)+\break  (1+\e) R \}$
for some fixed
$\e>0$. Again
by
H\"older's inequality we reduce this to bounding $\mathbb
{E}_{N-k,V}^N\exp\{{(1+\e')\mathcal{U}(t,x)} \}$ and
$\mathbb{E}_{N-k,V}^N\exp\{{(1+\e'')
R} \}$ for\vspace*{1pt} some $\e',\e''>0$.
Recall from (\ref{23}) that
\[
\mathcal{U}(x)=-\frac{1}{2\sqrt{2\pi}}\int\bigl\llvert\cir{u}_N(s,x)
\bigr\rrvert^2 \hat{h}(s)\,ds,
\]
where
\[
\cir{u}_N(s,x)=\sum_{j=1}^N
\cos(sx_j)-N\int\cos(st)\,d\mu(t)+\sqrt{-1}\sum
_{j=1}^N\sin(sx_j).
\]
For any $a$ and any $t_1,\ldots,t_k$ we get
%
%e62 #&#
\begin{eqnarray}\label{p25}
\qquad \mathcal{U}(t,x)&\leq&\frac{1}{2\sqrt{2\pi}}\int\Biggl\llvert\sum
_{j=k+1}^N\cos(sx_j)-(N-k)\int\cos(su)\,d
\mu(u)\nonumber
\\[-3pt]
&&\hspace*{76pt}{} +\sum_{j=1}^{k}\cos(st_j)-k
\int\cos(su)\,d\mu(u) \Biggr\rrvert^2 \hat{h}_-(s)\,ds
\nonumber
\\[-3pt]
&&{} +\frac{1}{2\sqrt{2\pi}}\int\Biggl\llvert\sum_{j=k+1}^N
\sin(sx_j)+ \sum_{j=1}^k
\sin(st_j)\Biggr\rrvert^2 \hat{h}_-(s)\,ds
\\[-2pt]
&\leq&\frac{1}{\sqrt{2\pi}}\int\Biggl\llvert\sum_{j=k+1}^N
\cos(sx_j)-(N-k)\int\cos(su)\,d\mu(u)\Biggr\rrvert^2
\hat{h}_-(s)\,ds
\nonumber
\\[-2pt]
&&{}+\frac{1}{\sqrt{2\pi}}\int\Biggl\llvert\sum_{j=k+1}^N
\sin(sx_j)\Biggr\rrvert^2 \hat{h}_-(s)\,ds+
\frac{5k^2}{\sqrt{2\pi
}}\int\hat{h}_-(s)\,ds,\nonumber
\end{eqnarray}
where we used the inequalities $(a + b)^2 \leq2 (a^2 + b^2)$ and
$\llvert\cos\rrvert,\llvert\sin\rrvert\leq1$.
From this we conclude as in the proof of Proposition~\ref{Fourier}
that
$\mathbb{E}_{N-k,V}^N \exp\{(1+\e')\mathcal
{U}(t, x) \}\leq C$ provided
that $\alpha_Q$ is large enough (which does not depend on $k$), and
$C$ does not depend on $t_1,\ldots,t_k$ or $N$. To\vspace*{1pt} see that
Theorem~\ref{Concentration} also applies for $P_{N-k,V}^N$ is obvious,
and for Proposition~\ref{KrSh} we use that
$P_{N-k,V}^N=P_{N-k,V,f}^{N-k}$ with $f(t):=kV(t)$, and the notation
introduced in (\ref{61}). Proposition~\ref{KrSh} is proved in
\cite{Shcherbina} also for the case of $P_{N,Q,f}$ for real-analytic
$Q$ and $f$, hence it can be applied as in the
proof of Proposition~\ref{Fourier}. We may now bound
$\mathbb{E}_{N-k,V}^N\exp\{{(1+\e'') R} \}$ as in the
arguments following (\ref{74}). Recall that
\[
R:=2\sum_{i\leq k, j>k}\log\llvert t_i-x_j
\rrvert+\log\biggl[F\bigl(0,N\mu(0)t\bigr)\frac
{Z_{N-k,V}^N}{Z_{N,V}} \biggr],
\]
where $F(a,t)$ was defined in (\ref{p23}). Using the same Jensen type
trick as in the proof of Lemma~\ref{lemma5}, we find that
$Z_{N-k,V}^N/Z_{N,V}\leq\exp\{CkN\}$ for some $C$. As in (\ref{74})
we get
%
%e63 #&#
\begin{eqnarray}\label{p27}
&&\mathbb{E}_{N-k,V}^N\exp\biggl\{{\bigl(2+2
\e''\bigr)\sum_{i\leq k, j>k}\log
\llvert t_i-x_j\rrvert} \biggr\}
\nonumber
\\
&&\qquad \leq\exp\biggl\{(N-k) \bigl(1+\e''\bigr)\sum
_{i\leq k}\log\bigl(1+t_i^2
\bigr) \biggr\}
\\
&&\quad\qquad{}\times \mathbb{E}_{N-k,V}^N\exp\biggl\{{\bigl(1+
\e''\bigr)\sum_{j>k}\log
\bigl(1+x_j^2\bigr) } \biggr\}.\nonumber
\end{eqnarray}
Analogously to (\ref{75}) we conclude that
$\mathbb
{E}_{N-k,V}^N\exp\{(2+2\e'')\times\break \sum_{j>k}\log(1+x_j^2)  \}
\leq\exp\{cN\}$ for some
$c>0$.
Using
(\ref{62}), it is straightforward to bound
%
%e64 #&#
\begin{eqnarray}\label{p26}
&&\exp\biggl\{(N-k) \bigl(2+2\e''\bigr)\sum
_{i\leq k}\log\bigl( 1+t_i^2 \bigr)+
\log\bigl[F\bigl(0,N\mu(0)t\bigr)Z_{N-k,V}^N/Z_{N,V}
\bigr] \biggr\}\hspace*{-23pt}
\nonumber\\[-8pt]\\[-8pt]
&&\qquad \leq\exp\Biggl\{-c_1N\sum_{i=1}^k
\bigl[V(t_i)-c_2\log\bigl(1+t_i^2
\bigr)\bigr]+CkN\Biggr\},\nonumber
\end{eqnarray}
where $c_1,c_2$ are absolute positive constants. Since $V$ is strictly
convex, this yields
\[
\mathbb{E}_{N-k,V}^N\exp\bigl\{{\bigl(1+
\e''\bigr) R} \bigr\}\leq e^{CN}
\]
and hence
\[
\mathbb{E}_{N-k,V}^N\exp\bigl\{{(1+\e)\mathcal{U}(t,x)+(1+
\e) R} \bigr\}\leq e^{C'N}
\]
for some $C,C'$. From (\ref{p24}), we get that for $L$ and $N$ large enough
\[
\mathbb{E}_{N-k,V}^N \bigl(\exp\bigl\{{\mathcal{U}(t,x)+R}
\bigr\} \mathbh{1}_{ ([-L,L]^N )^c}(x) \bigr)\leq e^{-C''N}
\]
for some $C''>0$ and all $t_1,\ldots,t_k$.

From (\ref{40}), (\ref{p25}) and (\ref{p27}) we also obtain
similarly as in Lemma~\ref{lemma5}
\[
\rho_{N,Q}^{h,k}(t_1,\ldots,t_k)\leq
\exp\Biggl\{{CN}-c_1N\sum_{i=1}^k
\bigl[V(t_i)-c_2\log\bigl(1+t_i^2
\bigr)\bigr]\Biggr\}
\]
for some positive $C,c_1,c_2$. As before, this implies that we can
assume all $t_1,\ldots,t_k$ to lie in some compact
set.

The second assertion of the lemma follows analogously from (\ref
{p24}), (\ref{p26}) and (\ref{p27}) with $\e''=0$.
\end{pf}
\begin{pf*}{Proof of Theorems~\ref{300} and~\ref{312}}
We first outline the main idea of the proof. Recall from (\ref{51}) that
\[
\mathcal{U}(x)=-(1/2) \Biggl(\sum_{i,j=1}^Nh(x_i-x_j)-
\bigl[h_\mu(x_i)+h_\mu(x_j)-h_{\mu\mu}
\bigr] \Biggr).
\]
Assume for a moment that $-h/2$ is positive semi-definite, or in other
words, the covariance function of a centered stationary
Gaussian process $(G_{t})_{t\in[-L,L]}$, that is, $-h(t-s)/2=\mathbb
{E}(G_tG_s)$. We may linearize the bivariate statistic
$-(1/2 )\sum_{i,j=1}^Nh(x_i-x_j)$ via
\[
\exp\Biggl\{-(1/2)\sum_{i,j=1}^Nh(x_i-x_j)
\Biggr\}=\mathbb{E}\exp\Biggl\{\sum_{j=1}^NG_{x_j}
\Biggr\},
\]
where $\mathbb{E}$ denotes expectation w.r.t. the underlying
probability measure.
% Now we obviously have
By definition we conclude that
%
%e65 #&#
\begin{equation}
\exp\bigl\{\mathcal{U}(x)\bigr\}=\mathbb{E}\exp\Biggl\{\sum
_{j=1}^NG_{x_j}-N\int G_\cdot\, d \mu\Biggr\},\label{p45}
\end{equation}
provided that $G_\cdot$ is a.s. integrable w.r.t. $\mu$. Since we %
%shall apply
would like to apply Corollary~\ref{Concentration2} %on
to the linear statistic in
(\ref{p45}), we need that $G_\cdot$ is sufficiently smooth with
probability one.
To see this, we use the well-known Karhunen--Lo\`eve
expansion of $G$.
By a classical result due to Mercer,
the covariance function $h$ admits an expansion, converging
uniformly on $[-L,L]$,
%
%e66 #&#
\begin{equation}
-h(t-s)/2=\sum_{i=1}^\infty
\lambda_i\theta_i(t)\overline{\theta
_i(s)},\label{p52}
\end{equation}
where $(\theta_i)_{i}$ denotes an orthonormal system of eigenfunctions
of the integral kernel~$h$ with real and positive
eigenvalues $(\lambda_i)_i$, that is,
\[
\int_{-L}^L -(1/2)h(t-s)\theta_i(s)\,ds=
\lambda_i\theta_i(t)\qquad\forall i.
\]

The Karhunen--Lo\`eve expansion of $G$ is then given by
%
%e67 #&#
\begin{equation}
G_t=\sum_{i=1}^\infty
\lambda_i^{1/2}\xi_i\theta_i(t),\label{p46}
\end{equation}
where $(\xi_i)_i$, $\xi_i:=(\lambda_i)^{-1/2}\int_{-L}^L\theta
_i(t)G_t\,dt$, are
independent standard normal variables. The convergence in
(\ref{p46}) is a.s. uniform on the compact interval $[-L,L]$; see
\cite{AdlerTaylor}, Theorem 3.1.2.
The a.s. continuity of $G_t$ used for this theorem follows, for
example, from the Kolmogorov--Chentsov theorem (\cite{Kallenberg}, Theorem
3.23).
Since $h$ is analytic on some domain containing the
compact set, say $A:=[-L,L]\times[-\delta,\delta]\subset
\mathbb{C}$, $\delta>0$, its eigenfunctions (with nonzero
eigenvalues) are
analytic on~$A$. Hence the uniform convergence in (\ref{p46}) implies that
$G_w, w\in A$ is analytic with probability one. Furthermore, recall
that the
derivative process $(G_t')_{t\in[-L,L]}$ of $G$ is a centered
(real-valued) Gaussian process with covariance function $h''/2$; see,
for example, \cite{Adler}, Theorem 2.2.2.

To summarize, if $-h$ is positive semi-definite, $\mathcal{U}$ admits
the linearization (\ref{p45}) in terms of linear statistics
with random test functions which fulfill the prerequisites of Corollary
\ref{Concentration2} if we restrict ourselves to a compact
$[-L,L]$. In the following we sketch the main strategy in this case.
Let $k\in\mathbb{N}$ be fixed. Eventually we will prove
%
%e68 #&#
\begin{equation}
\lim_{N\to\infty}\rho_{N,Q}^{h,k} \biggl(a+
\frac{t_1}{N\mu
(a)},\ldots,a+\frac{t_k}{N\mu(a)} \biggr)-\mathbb{S}^k(t)=0
\label{60}
\end{equation}
locally uniformly, where
\[
\mathbb{S}^k(t):=\mu(a)^k\det\biggl[\frac{\sin(\pi(t_i-t_j) )}{\pi
(t_i-t_j)}
\biggr]_{1\leq i,j\leq k}.
\]
By the boundedness of $\mathbb{E}_{N,V}e^{\mathcal{U}}$ (Proposition
\ref{Fourier}) and Lemma~\ref{lemma9}, (\ref{60})
converges to zero if and only if
\[
\mathbb{E}_{N,V,L}e^{\mathcal{U}}\rho_{N,Q}^{h,k}
\biggl(a+\frac
{t_1}{N\mu(a)},\ldots,a+\frac{t_k}{N\mu(a)} \biggr)-\mathbb
{E}_{N,V,L}e^{\mathcal{U}}\mathbb{S}^k(t)
\]
tends to $0$, where the $L>0$ was introduced in Lemma~\ref{lemma9}.
But this means, using (\ref{56}), (\ref{40}) and the
abbreviation
$R_{a,L}$,
which denotes a version of $R_a$ which is truncated to $[-L,L]$ [see
(\ref{p53})] and Lemma~\ref{lemma9} that
%
%e69 #&#
\begin{equation}\label{p54}
\mathbb{E}_{N-k,V,L}^N\exp\bigl\{{\mathcal{U}(t,x)+R_{a,L}}
\bigr\} -\mathbb{E}_{N,V,L}\exp\{{\mathcal{U}} \}\mathbb
{S}^k(t)\to0
\end{equation}
as $N\to\infty$. The linearization procedure then gives
%
%e70 #&#
\begin{eqnarray}\label{652}
&& \mathbb{E}_{N-k,V,L}^N\exp\bigl\{{\mathcal{U}(t,x)+R_{a,L}}
\bigr\} -\mathbb{E}_{N,V,L}\exp\{{\mathcal{U}} \}\mathbb
{S}^k(t)\nonumber
\\
&&\qquad =\mathbb{E} \Biggl[\mathbb{E}_{N-k,V,L}^N \exp\Biggl\{{\sum
_{j=1}^NG_{(t,x)_j}
+R_{a,L}} \Biggr\}
\\
&&\hspace*{68pt}{} -\mathbb{E}_{N,V,L}\exp\Biggl
\{{\sum_{j=1}^NG_{x_j}} \Biggr\}
\mathbb{S}^k(t) \Biggr].\nonumber
\end{eqnarray}
We find similarly as in (\ref{40}) that
%
%e71 #&#
\begin{eqnarray}\label{p302}
&& \Biggl(\mathbb{E}_{N,V,L}\exp\Biggl\{{\sum
_{j=1}^NG_{x_j}} \Biggr\}
\Biggr)^{-1}\mathbb{E}_{N-k,V,L}^N \exp\Biggl\{{\sum
_{j=1}^NG_{(t,x)_j}
+R_{a,L}} \Biggr\}
\nonumber\\[-8pt]\\[-8pt]
&&\qquad =\rho_{N,V,G_\cdot,L}^k \biggl(a+\frac{t_1}{N\mu(a)},\ldots,a+
\frac{t_k}{N\mu(a)} \biggr),\nonumber
\end{eqnarray}
where $P_{N,V,G_\cdot,L}$ denotes the determinantal ensemble on
$[-L,L]^N$ with external field
$\exp\{-NV(t)+G_t \}$.

With representation (\ref{p302}), we can use the bulk universality
of $P_{N,V,G_\cdot,L}$ to show convergence of
%
%e72 #&#
\begin{equation}
\mathbb{E}_{N-k,V,L}^N \exp\Biggl\{{\sum
_{j=1}^NG_{(t,x)_j} +R_{a,L}}
\Biggr\}-\mathbb{E}_{N,V,L}\exp\Biggl\{{\sum
_{j=1}^NG_{x_j}} \Biggr\}
\mathbb{S}^k(t)\label{653}
\end{equation}
to $0$ almost surely. To show that convergence to $0$ also holds for
the expectation, we will bound (\ref{653}) in terms of
$G_\cdot$. Here we can use that $G$ is a Gaussian process and
quantities like $\|G_\cdot\|_\infty$ and
$\|G'_\cdot\|_\infty$ have sub-Gaussian tails.

We now turn to the detailed proof. As $-h$ is in general not
positive semi-definite, we may extend the previous case by means of the
following argument.\vspace*{1.5pt} Recall the decomposition of $\hat{h}$
into nonnegative functions $\hat{h}=(\hat{h})_+-(\hat{h})_-$. By setting
$h^+:=\widehat{(\hat{h})_+}$, $h^-:=\widehat{(\hat{h})_-}$,
we get a
decomposition
$h=h^+-h^-$ of $h$ into positive semi-definite, real-analytic functions.
Define for a complex parameter $z\in\mathbb{C}$
%
%e73 #&#
%e74 #&#
\begin{eqnarray}
\mathcal{U}_z(x)&:=&\frac{z}{2} \Biggl(\sum
_{i,j=1}^Nh^+(x_i-x_j)-
\bigl[h^+_\mu(x_i)+h^+_\mu(x_j)-h^+_{\mu\mu}
\bigr] \Biggr)\label{p47}
\\
&&{} +\frac{1}{2} \Biggl(\sum_{i,j=1}^Nh^-(x_i-x_j)-
\bigl[h^-_\mu(x_i)+h^-_\mu(x_j)-h^-_{\mu\mu}
\bigr] \Biggr)\label{p48}.
\end{eqnarray}
Note that $\mathcal{U}_{-1}=\mathcal{U}$.
Similar to (\ref{p54}), we have to show that for $z=-1$,
\[
\mathbb{E}_{N-k,V,L}^N\exp\bigl\{{\mathcal{U}_{z}(t,x)+R_{a,L}}
\bigr\}-\mathbb{E}_{N,V,L}\exp\{{\mathcal{U}_{z}} \}\mathbb
{S}^k(t)\to0
\]
as $N\to\infty$.
As the linearization procedure only works for nonnegative $z$, we shall
use the following result, known as Vitali's convergence
theorem, which can be found, for example, in
\cite{Titchmarsh}.
%
%th6.4 #&#
\begin{thrm}[(Vitali's convergence theorem)]\label{Vitali}
Let $f_n(z)$ be a sequence of analytic functions on a region $D\subset
\mathbb{C}$ with $\llvert f_n(z)\rrvert\leq M$ for all $n$
and all $z\in D$.
Assume that $\lim_{n\to\infty}f_n(z)$ exists for a set of $z$ having
a limit point in $D$. Then $\lim_{n\to\infty}f_n(z)$ exists
for all $z$ in the interior of $D$ and the limit is an analytic
function in $z$.
\end{thrm}

We will apply Vitali's convergence theorem to the sequence (in $N$) of
the following analytic functions of $z$:
%
%e75 #&#
\begin{equation}
\quad W_{N,z}(a,t):=\mathbb{E}_{N-k,V,L}^N\exp\bigl\{{
\mathcal{U}_{z}(t,x)+R_{a,L}} \bigr\}-\mathbb{E}_{N,V,L}
\exp\{{\mathcal{U}_{z}} \}\mathbb{S}^k(t)\label{90}.
\end{equation}
Introduce the domain $D:= \{z=x+iy\in\mathbb{C}\dvtx  x,y\in
\mathbb{R}, x< C(\alpha_Q) \}$, where $C(\alpha_Q)>0$ is a
sufficiently small constant such that the following quantity is bounded
by some constant $C$:
\[
\mathbb{E}_{N,V,L}\exp\{{\mathcal{U}_{C(\alpha_Q)}} \}\leq C
\]
(the existence of such constants follows from the proof of Proposition~\ref{Fourier}).
First we shall show uniform boundedness of $W_{N,z}(a,t)$ for all
$N,a,t$ and $z\in D$. By the definition of $\mathcal{U}_z$
in (\ref{p47}) and the positivity
of (\ref{p48}) and (\ref{p47}) for positive $z$ (being variances of
Gaussian random variables) it is clear that it suffices to
bound $W_{N,z}(a,t)$
for real,
positive $z$, since for negative real parts of $z$ the boundedness of
$W_{N,z}(a,t)$ is obvious. Hence we restrict ourselves to
$0\leq z< C(\alpha_Q)$ only. Let $G^+$ and $G^-$ denote two
independent, centered and stationary
Gaussian processes on a
probability space $(\Omega, \mathcal{A},P)$ indexed by
$A:=[-L,L]\times[-\e,\e]\subset\mathbb{C}$ with covariance
functions $(z/2)h^+$ and $h^-/2$, respectively, where $h^+$ and $h^-$ are
analytic on
$A$. Writing $G_t=G_t^+-\int G^+_\cdot \,d\mu+G^-_t -\int G^-_\cdot \,d\mu
$ and denoting by $\mathbb{E}$ the expectation w.r.t. $P$,
we can
rewrite
%
%e76 #&#
\begin{eqnarray}\label{65}
&&\mathbb{E}_{N-k,V,L}^N\exp\bigl\{{\mathcal{U}_{z}(t,x)+R_{a,L}}
\bigr\}-\mathbb{E}_{N,V,L}\exp\{{\mathcal{U}_{z}} \}\mathbb
{S}^k(t)\hspace*{-9pt}
\nonumber\\[-4pt]\\[-12pt]\nonumber
&&\qquad =\mathbb{E} \Biggl[\mathbb{E}_{N-k,V,L}^N \exp\Biggl\{{\sum
_{j=1}^NG_{(t,x)_j}
+R_{a,L}} \Biggr\}
-\mathbb{E}_{N,V,L}\exp\Biggl
\{{\sum_{j=1}^NG_{x_j}} \Biggr\}
\mathbb{S}^k(t) \Biggr].\hspace*{-9pt}
\end{eqnarray}
Similar to (\ref{p302}), we have
%
%e77 #&#
\begin{eqnarray}\label{p30}
&& \Biggl(\mathbb{E}_{N,V,L}\exp\Biggl\{{\sum
_{j=1}^NG_{x_j}} \Biggr\}
\Biggr)^{-1}\mathbb{E}_{N-k,V,L}^N \exp\Biggl\{{\sum
_{j=1}^NG_{(t,x)_j}
+R_{a,L}} \Biggr\}
\nonumber\\[-8pt]\\[-8pt]
&&\qquad =\rho_{N,V,G_\cdot,L}^k \biggl(a+\frac{t_1}{N\mu(a)},\ldots,a+
\frac{t_k}{N\mu(a)} \biggr),\nonumber
\end{eqnarray}
where $P_{N,V,G_\cdot,L}$ denotes the determinantal ensemble on
$[-L,L]^N$ with external field
$\exp\{-NV(t)+G_t^++G_t^- \}$.

Fix compact sets $E\subset\mathbb{R}^k$ and $I\subset\operatorname{supp}
\mu^\circ$. We have
%
%e78 #&#
\begin{eqnarray}\label{p49}
&&\sup_{t\in E,a\in I} \Biggl|\mathbb{E} \Biggl[\mathbb{E}_{N-k,V,L}^N
\exp\Biggl\{{\sum_{j=1}^NG_{(t,x)_j}
+R_{a,L}} \Biggr\}\nonumber
\\
&&\hspace*{69pt}{} -\mathbb{E}_{N,V,L}\exp\Biggl\{{\sum
_{j=1}^NG_{x_j}} \Biggr\}
\mathbb{S}^k(t) \Biggr] \Biggr|
\\
&&\qquad \leq\mathbb{E}\sup_{t\in E,a\in I} \Biggl|\mathbb{E}_{N-k,V,L}^N
\exp\Biggl\{{\sum_{j=1}^NG_{(t,x)_j}
+R_{a,L}} \Biggr\}\nonumber
\\
&&\hspace*{97pt}{} -\mathbb{E}_{N,V,L}\exp\Biggl\{{\sum
_{j=1}^NG_{x_j}} \Biggr\}
\mathbb{S}^k(t) \Biggr|.\nonumber
\end{eqnarray}
Since (\ref{p30}) converges by Theorem~\ref{thrmLubinsky} to $\mathbb
{S}^k(t)$ locally uniformly and
the term $\mathbb{E}_{N,V,L}\exp\{{\sum_{j=1}^NG_{x_j}} \}$
is bounded in $N$ by Corollary~\ref{Concentration2} and bounded
away from~$0$ by Proposition~\ref{KrSh} and Lemma~\ref{lemma9}, we
see that the term
%
%e79 #&#
\begin{equation}\label{p35}
\sup_{t\in E,a\in I} \Biggl|\mathbb{E}_{N-k,V,L}^N \exp
\Biggl\{{\sum_{j=1}^NG_{(t,x)_j}
+R_{a,L}} \Biggr\}-\mathbb{E}_{N,V,L}\exp\Biggl\{{\sum
_{j=1}^NG_{x_j}} \Biggr\}
\mathbb{S}^k(t) \Biggr|\hspace*{-25pt}
\end{equation}
converges to $0$ a.s. w.r.t. $P$. To show convergence of (\ref{p49})
to 0, it remains to show
that (\ref{p35}) is uniformly integrable w.r.t. $P$. We first consider
the term
$\mathbb{E}_{N,V,L}\exp\{{\sum_{j=1}^NG_{x_j}} \}$. In view
of Corollary~\ref{Concentration2}, we need to determine the
distribution of
the Lipschitz constant of $G^++G^-$ and of
%
%e80 #&#
\begin{equation}
\bigl\|G^++G^-\bigr\|_\infty+\bigl\|\bigl(G^++G^-\bigr)^{(3)}
\bigr\|_\infty\label{p50}
\end{equation}
on $[-L,L]$.
The derivative processes $(G^+)'$ and $(G^-)'$ are Gaussian with
covariance functions $-(z/2)(h^+)''$ and
$-(h^-)''/2$, respectively. Furthermore, it is well known that $\sup
_{t\in[-L,L]}\llvert
G^+_t\rrvert$ and $\sup_{t\in[-L,L]}\llvert
G^-_t\rrvert$ are sub-Gaussian with certain means and variances
$-(z/2)(h^+)''(0)$ and $-(h^-)''(0)/2$, respectively. By the same
argument, $\|G^++G^-\|_\infty$ and $\|(G^++G^-)^{(3)}\|_\infty$ are
sub-Gaussian
with certain means and the variances given in terms of derivatives of
$(h^+)$ and $(h^-)$. For a reference, see, for example, \cite{AdlerTaylor}, Theorem
2.1.1. From
the sub-Gaussianity of these quantities and Corollary~\ref{Concentration2}, it is easy to see that
%
%e81 #&#
\begin{equation}
\mathbb{E}_{N,V,L}\exp\Biggl\{{\sum_{j=1}^NG_{x_j}}
\Biggr\},\label{p51}
\end{equation}
has a $P$-integrable dominating function,
provided that $\alpha_Q$ (and hence $\alpha_V$) is large enough. Note
that the estimates above are uniform in $z$ varying in a small
interval. It remains to show that
%
%e82 #&#
\begin{equation}
\label{66} \mathbb{E}_{N-k,V,L}^N \exp\Biggl\{{\sum
_{j=1}^NG_{(t,x)_j}
+R_{a,L}} \Biggr\}
\end{equation}
is uniformly integrable and bounded in $z$ for $z$ varying in a small
interval. To this end we use that (\ref{66}) is
equal to
\[
\mathbb{E}_{N,V,L}\exp\Biggl\{{\sum_{j=1}^NG_{x_j}}
\Biggr\}\rho_{N,V,G_\cdot,L}^k \biggl(a+\frac{t_1}{N\mu(a)},\ldots,a+
\frac
{t_k}{N\mu(a)} \biggr).
\]
As in the proof of Theorem~\ref{313}, we get
\begin{eqnarray*}
&& \rho_{N,V,G_\cdot,L}^k \biggl(a+\frac{t_1}{N\mu(a)},\ldots,a+\frac
{t_k}{N\mu(a)} \biggr)
\\
&&\qquad \leq C^k\prod
_{j=1}^k \rho_{N,V,G_\cdot,L}^1 \biggl(
a+\frac{t_j}{N\mu(a)} \biggr),
\end{eqnarray*}
where $C$ is such that $C\geq N/(N-k)$. By Lemma~\ref{cor2} we have
\[
\rho_{N,V,G_\cdot,L}^1 \biggl(a+\frac{t_j}{N\mu(a)} \biggr)\leq\rho
_{N,V,L}^1 \biggl(a+\frac{t_j}{N\mu(a)} \biggr)e^{2\|G_\cdot\|_\infty},
\]
where\vspace*{-1pt} $\|G_\cdot\|_\infty:=\sup_{t\in[-L,L]}\llvert G_t\rrvert$.
Bulk universality for $k=1$ gives that\break $\rho_{N,V,L}^1(
a+\frac{t_j}{N\mu(a)})$ converges (locally) uniformly toward the
bounded function~$\mu(a)$. We conclude that there is a constant
$C>0$ such that for $t_1,\ldots,t_k\in E$, $a\in I$ we have
\[
\rho_{N,V,G_\cdot,L}^k \biggl(a+\frac{t_1}{N\mu(a)},\ldots,a+
\frac
{t_k}{N\mu(a)} \biggr)\leq Ce^{2k\|G_\cdot\|_\infty}.
\]
As $\|G_\cdot\|_\infty$ is sub-Gaussian, we get in combination with
(\ref{p51}) that (\ref{p35}) is uniformly integrable w.r.t. $P$,
provided that $\alpha_Q$ is large enough. It is clear that
this bound is uniform in $z\in[0,\e)$ for some small $\e>0$.

To summarize, we have shown that (\ref{p49}) converges to $0$ for
(small) positive $z$, or in other terms, locally uniform
convergence in $a$ and $t$ of $W_{N,z}(a,t)$ (for small positive $z$)
as $N\to\infty$. We have also shown uniform boundedness of
$W_{N,z}(a,t)$ for arbitrary $N, a, t$ and
$z\in(-\infty,\e)\times\mathbb{R}\subset\mathbb{C}$ and as
locally uniform convergence implies pointwise convergence, we get by Vitali's
convergence theorem that the sequence (in $N$) of functions
$W_{N,z}(a,t)$ converges to 0 for $z=-1$ pointwise in $a$ and $t$.
To get locally uniform convergence in $t$ and $a$
for $z=-1$, recall that by Arzel\`a--Ascoli's theorem, a sequence of
continuous functions on a compact set has a uniformly
converging
subsequence if and only if the sequence is uniformly bounded and
equicontinuous. Thus it remains to
show that $(W_{N,z}(a,t))_N$
is equicontinuous in $a$ and $t$ (boundedness has
already been shown). As the convergence of $W_{N,z}(a,t)$ is uniform in
$a,t$ for small positive $z$, Arzel\`a--Ascoli's theorem
implies equicontinuity (in $a,t$) of
$(W_{N,z}(a,t))_N$ for small positive $z$. To see that this implies
equicontinuity (in $a,t$) of
$(W_{N,z}(a,t))_N$ also for $z=-1$, observe that a (real-valued)
sequence of functions $(f_N)_N$ on some compact $K\subset\mathbb
{R}^d$ is
equicontinuous in $x\in K$ if and only if for each sequence
$(x_m)_m\subset K$, $\lim_{m\to\infty}x_m=x$ and each sequence
$(N_m)_m\subset\mathbb{N}$ we have $\lim_{m\to\infty}
f_{N_m}(x_m)-f_{N_m}(x)=0$. Using this characterisation, equicontinuity for
$z=-1$ is easily seen by applying Vitali's convergence theorem to
deduce $\lim_{m\to\infty} W_{N_m,-1}(a_m,t_m)=0$ from
$\lim_{m\to\infty} W_{N_m,z}(a_m,t_m)=0$ for small positive $z$.
This completes the proof of Theorem~\ref{312}.

To prove Theorem~\ref{300}, take $g\dvtx  \mathbb{R}^k \longrightarrow
\mathbb{R}$ bounded and continuous. With the same arguments as
above, we arrive in analogy
to (\ref{65})--(\ref{p30}) at proving
\begin{eqnarray*}
&& \mathbb{E} \Biggl[\mathbb{E}_{N,V,L}\exp\Biggl\{{\sum
_{j=1}^NG_{x_j}} \Biggr\}\int
_{\mathbb{R}^k}\rho_{N,V,G_\cdot,L}^k (t_1,\ldots,t_k ) g(t_1,\ldots,t_k)\,dt_1
\cdots dt_k
\\
&&\hspace*{28pt}{}-\mathbb{E}_{N,V,L}\exp\Biggl\{{\sum_{j=1}^NG_{x_j}}
\Biggr\}\int_{\mathbb{R}^k} g(t_1,\ldots,t_k)
\mu(t_1)\cdots\mu(t_k)\,dt_1\cdots dt_k \Biggr]\to0.
\end{eqnarray*}
All the boundedness and integrability arguments above for\break
$\mathbb {E}_{N,V,L}\exp\{{\sum_{j=1}^NG_{x_j}} \}$ can be used again.
The convergence of\break
$\int_{\mathbb{R}^k}\rho_{N,V,G_\cdot}^k (t )
g(t)\,dt$ toward $\int g(t)\mu(t_1)\cdots\mu(t_k)\,dt$ is given by \cite{Johansson98}, Theorem 2.1. Lemma~\ref{lemma9} enables us to transfer
Johansson's result to the correlation function $\rho_{N,V,G_\cdot,L}^k$. This finishes the proof of Theorem
\ref{300}.
\end{pf*}

\setcounter{thrm}{0}
%sA #&#
\begin{appendix}
\section*{Appendix: Equilibrium measures with external fields}\label{equmeas}
In this appendix, we recall some results about equilibrium measures,
mainly from the book by Saff and Totik \cite{SaffTotik}, Section~I.1. The following can be
found in~\cite{SaffTotik}, Section~I.1.

Let $\mathcal{M}^1(\Sigma)$ denote the set of Borel probability
measures on a set $\Sigma$. Define for $\Sigma\subset\mathbb{C}$
compact the \textit{logarithmic energy of $\mu\in\mathcal{M}^1(\Sigma)$} as
%
%eA.1 #&#
\begin{equation}\label{47}
I(\mu):=\int\!\!\int\log\llvert z-t\rrvert^{-1}\,d\mu(z)\,d\mu
(t)
\end{equation}
and the \textit{energy $V$ of $\Sigma$} by $V:=\inf_{\mu\in\mathcal
{M}^1(\Sigma
)}I(\mu)$.
It turns out that $V$ is finite or~$\infty$ and in the finite case
there is a unique measure $\omega_\Sigma$ which minimizes (\ref{47}).
This measure $\omega_\Sigma$ is
called \textit{equilibrium measure of $\Sigma$} and the quantity
$\operatorname{cap}
(\Sigma):=e^{-V}$
is called \textit{capacity of $\Sigma$}. For an arbitrary Borel set
$\Sigma$, we define the capacity of $\Sigma$~as
\[
\operatorname{cap}(\Sigma):=\sup\bigl\{\operatorname{cap}(K)\dvtx  K\subset
\Sigma
\mbox{ compact}\bigr\}.
\]

%leA.1 #&#
\begin{lemma}\label{lemma3}
If $\Sigma=[-l,l]$, $l\geq0$, then $\operatorname{cap}(\Sigma)=l/2$
and the
equilibrium measure is the arcsine distribution with support $[-l,l]$,
\[
d\o_\Sigma(t)=\frac{1}{\pi\sqrt{l^2-t^2}}\,dt,\qquad t\in[-l,l].
\]
$\o_\Sigma$ has mean $0$ and variance $l^2/2$.
\end{lemma}
\begin{pf}
See \cite{SaffTotik}, Section~I.1.
\end{pf}
%
%
%deA.2 #&#
\begin{defi}\label{defi1} Let $\Sigma\subset\mathbb{R}$ be closed.
Let $Q\dvtx  \Sigma \longrightarrow [0,\infty]$ satisfy:
\begin{longlist}[(a)]
\item[(a)] $Q$ is lower semicontinuous;
\item[(b)] $\Sigma_0:=\{t\in\Sigma\dvtx  Q(t)<\infty\}$ has positive
capacity;
\item[(c)] if $\Sigma$ is unbounded, then $\lim_{\llvert
t\rrvert\to\infty,t\in\Sigma}Q(t)-\log\llvert
t\rrvert=\infty$.
\end{longlist}
If $Q$ satisfies these properties, we call it \textit{external field} on
$\Sigma$ and $W=e^{-Q}$ its corresponding \textit{weight function}.
\end{defi}
Furthermore, define for $\mu\in\mathcal{M}^1(\Sigma)$ the \textit{energy functional}
%
%eA.2 #&#
\begin{equation}
I_Q(\mu):=\int Q(t)\,d\mu(t)+\int\!\!\int\log\llvert s-t\rrvert
^{-1}\,d\mu(s)\,d\mu(t).\label{120}
\end{equation}

%
%reA.3 #&#
\begin{remark}\label{remark1}
In \cite{SaffTotik} the authors define the energy functional to be (in
our notation) $I_{2Q}$ instead of $I_Q$. It is more
convenient for our purposes to use this definition. We note that under
this change qualitative
results from \cite{SaffTotik} remain the same but quantitative results
involving $Q$ have to be changed by a factor $2$ or $1/2$, respectively.
\end{remark}
$I_Q(\mu)$ might be $\infty$, but the following theorem holds. The
support of a measure $\mu$ will be denoted as $\operatorname{supp}(\mu)$.
%
%thA.4 #&#
\begin{thrm}\label{thrm4}
Let $Q$ be an external field on $\Sigma$.
\begin{longlist}[(a)]
\item[(a)]
There is a unique probability measure $\mu_Q\in\mathcal{M}^1(\Sigma
)$ with
%
%eA.3 #&#
\begin{equation}
I_Q(\mu_Q)=\inf_{\mu\in\mathcal{M}^1(\Sigma)}I_Q(
\mu)\label{43}.
\end{equation}
\item[(b)]
$\mu_Q$ has a compact support.
\item[(c)] Let $\widetilde{Q}$ be an external field on $\Sigma$ such that
$\widetilde{Q}=Q$ on a compact set $K$ with $\operatorname{supp}(\mu
_Q)\subset K$ and
$\widetilde{Q}(t)=\infty$ for $t\notin K$. Then
$\mu_{\widetilde{Q}}=\mu_Q$.
\end{longlist}
\end{thrm}
\begin{pf}
Statements (a) and (b) can be found in \cite{SaffTotik}, Theorem I.1.3, (c) follows from \cite{SaffTotik}, Theorem I.3.3
(also see the
remark on
page 48 in \cite{SaffTotik}).
\end{pf}
$\mu_Q$ is called the \textit{equilibrium measure} for $Q$.
The next theorem summarizes properties of the \textit{logarithmic potential}
\[
U^\mu(z):=\int\log\llvert z-t\rrvert^{-1}\,d\mu(t).
\]
%

%thA.5 #&#
\begin{thrm}\label{thrm5}
\textup{(a)}~Let $Q$ and $\widetilde{Q}$ be external fields on $\Sigma
$ such that $\llvert Q-\widetilde{Q}\rrvert\leq\e$ on $\Sigma$.
Then for
all $z\in\mathbb{C}$,
\[
\bigl\llvert U^{\mu_Q}(z)-U^{\mu_{\widetilde{Q}}}(z)\bigr\rrvert\leq
2\e.
\]

\textup{(b)} Let $K\subset\mathbb{R}$ be compact and $(\mu_n)_n$ be a
sequence in $\mathcal{M}^1(K)$ converging weakly to a probability measure
$\mu$. Then for a.e. $z\in\mathbb{C}$ (w.r.t. the Lebesgue measure
on $\mathbb{C}$),
\[
\liminf_{n\to\infty} U^{\mu_n}(z)=U^\mu(z).
\]

\textup{(c)} If $\mu$ and $\nu$ are two compactly supported
probability measures and their logarithmic potentials $U^\mu$ and
$U^\nu$
coincide almost everywhere on $\mathbb{C}$, then $\mu=\nu$.
\end{thrm}

\begin{pf}
Statement (a) is contained in \cite{SaffTotik}, Corollary I.4.2,
statement (b) is~\cite{SaffTotik}, Theorem I.6.9, and assertion (c) is
\cite{SaffTotik}, Corollary II.2.2.
\end{pf}
%
%We also need a characterization of the support of the equilibrium
%measure.
%
%thA.6 #&#
\begin{thrm}\label{thrm3}
Let $Q$ be an external field on $\Sigma$.
\begin{longlist}[(a)]
\item[(a)] For a compact set $K$ of positive capacity, define the functional
\[
F_Q(K):=\log\operatorname{cap}(K)-2\int Q \,d\omega_K.
\]
For any compact $K$ of positive capacity, we have $F_Q(K)\leq
F_Q(\operatorname{supp}
(\mu_Q))$. Furthermore, if $K$ is compact and of positive
capacity and such that $F_Q(K)= F_Q(\operatorname{supp}(\mu_Q))$, then
$\operatorname{supp}(\mu_Q)\subset K$.
\item[(b)] If $Q$ is convex, then $\operatorname{supp}(\mu_Q)$ is an
interval.
\item[(c)] If $Q$ is even, then $\operatorname{supp}(\mu_Q)$ is even.
\end{longlist}
\end{thrm}
\begin{pf}
For statement (a), see \cite{SaffTotik}, Theorem IV.1.5, for
statements (b)~and~(c), see \cite{SaffTotik}, Theorem IV.1.10.
\end{pf}
%
%The last fact is about existence and properties of a continuous
%density of the equilibrium measure.

%
%thA.7 #&#
\begin{thrm}\label{thrm6}
\textup{(a)}~Let $Q$ be an external field on $\Sigma$. If $Q$ is finite on
$\operatorname{supp}
(\mu_Q)$ and locally of class $C^{1+\e}$ for some $\e>0$ (which
means that $Q$ is continuously differentiable and the derivative $Q'$
is H\"older continuous with parameter $\e$), then $\mu_Q$ has
a continuous density on the interior of $\operatorname{supp}(\mu_Q)$.

\textup{(b)} If $Q$ has two Lipschitz derivatives and is strictly
convex, then\break $\operatorname{supp}(\mu_Q)=:[a,b]$ and the density of
$\mu_Q$ can be
represented as
%
%eA.4 #&#
\begin{equation}
\label{57} \frac{d \mu(t)}{dt}=r(t)\sqrt{(t-a) (b-t)}\mathbh{1}_{[a,b]}(t),
\end{equation}
where $r$ can be extended into an analytic function on a domain
containing $[a,b]$ and $r(t)>0$ for $t\in[a,b]$. In particular, the
density is positive on $(a,b)$.
\end{thrm}

\begin{pf}
Statement (a) is \cite{SaffTotik}, Theorem IV.2.5, and for assertion
(b), see, for example, the appendix of the paper by McLaughlin and Miller~\cite{McLaughlinMiller08}.
\end{pf}
\end{appendix}

% AOS,AOAS: If there are supplements please fill\dvtx
% \sname{Supplement A}
% \stitle{Title}
% \slink[doi]{10.1214/00-AOASXXXXSUPP}
% \sdatatype{.pdf}"
% \sdescription{Some text}

\section*{Acknowledgment}
The second author is grateful to L.~A.~Pastur for a helpful discussion.

%suskaldyti doi

% imsref loaded by linak, 2014-01-15 12:50:44
% imsref loaded by linak, 2014-07-30 10:36:28
% imsref loaded by linak, 2014-07-30 10:36:50
% imsref loaded by linak, 2014-07-30 10:40:23

\printaddresses

\end{document}